\UseRawInputEncoding
\documentclass[reqno, 11pt]{article} 
\usepackage{amsthm,amsmath,amsfonts,amssymb,mathtools,bm,framed,fixltx2e}
\usepackage{bbm}
\usepackage{mathbbol}
\usepackage{mathrsfs}
\usepackage{amscd}
\usepackage{enumerate}
\usepackage{lineno}

\usepackage[top=4cm, bottom=3cm, left=3.2cm, right=3.2cm]{geometry}

\newtheorem{thm}{Theorem}
\newtheorem{cor}{Corollary}
\newtheorem{lem}{Lemma}
\newtheorem{prop}{Proposition}
\theoremstyle{definition}
\newtheorem{defn}{Definition}
\theoremstyle{remark}
\newtheorem{rem}{Remark}
\theoremstyle{Hypothesis}

\usepackage{color}
\usepackage{ulem}

\numberwithin{equation}{section} \numberwithin{lem}{section}
\numberwithin{thm}{section} \numberwithin{prop}{section}
\numberwithin{cor}{section} \numberwithin{rem}{section}\numberwithin{hyp}{section}

\title
{Propagation of chaos for the parabolic-parabolic Keller-Segel model with a logarithmic cut-off}

\author{ Li {\sc Chen}$^2$
\and Shu Wang$^1$
\and Rong  Yang$^1$
}

\begin{document}
\maketitle
\begin{center}
{\footnotesize
1-College of Applied Sciences, Beijing University of Technology, Ping Le Yuan 100, Chaoyang District, Beijing, 100124, People's Republic of China\\
 \smallskip
2-University of Mannheim, Department of Mathematics,\\ 68131, Mannheim, Germany
}
\end{center}
\maketitle
\date{}

\begin{abstract}
In this paper, we study propagation of chaos for the parabolic-parabolic Keller-Segel model with a logarithmic cut-off by establishing a rigorous  convergence analysis from a stochastic particle system to the parabolic-parabolic Keller-Segel (KS) equation for any dimension case. Under the assumption that the initial data are independent and identically distributed (i.i.d.) with a common probability density function
$\rho_0$, we rigorously prove the propagation of chaos for this interacting system with a cut-off parameter $\varepsilon\sim (\ln N)^{-\frac{2}{d+2}}$: when $N\rightarrow \infty$, the joint distribution of the particle system is $f$-chaotic and the measure $f$ possesses a density  which is a weak solution to the mean-field parabolic-parabolic KS equation.

\end{abstract}

\section{Introduction}
In this paper,  we study a convergence of particle system to the standard parabolic-parabolic KS system for any  $d\geq 1$ dimension:
\begin{eqnarray}\label{diffaggKSeq}
\left\{\begin{aligned}{}
&\partial_t\rho-\triangle \rho +\nabla\cdot(\rho\nabla c)=0, &&x \in \mathbb{R}^{d}, ~t\geq0, \\
&\partial_t c-\triangle c+\lambda c=\rho, && x \in \mathbb{R}^{d}, ~t\geq0,\\
&\rho(x,0)=\rho_0(x), \hspace{0.5cm} c(x,0)=c_0(x),&& x \in \mathbb{R}^{d},
\end{aligned}\right.
\end{eqnarray}
where $\rho(x,t)$ represents the density of bacteria and $c(x,t)$ represents the chemical substance concentration. The constant $\lambda\geq0$ and the initial data $\rho_0(x)$ is given by a non-negative density. We refer Corrias et al. \cite{CEM2014}, Perthame \cite{P2004} and references therein
for theoretical results on this system of PDEs and applications to biology. The parabolic-parabolic KS system \eqref{diffaggKSeq} has been obtained as a diffusion limit from a kinetic equation
by Erban et al. \cite{EO2007, OH2002}. Mean--field limit of a particle approximation of the one-dimensional parabolic--parabolic KS model without smoothing has been obtained in \cite{JTT2017}. Recently, the parabolic-parabolic KS system \eqref{diffaggKSeq} in two and three space dimensions has also been obtained as a vanishing cross-diffusion limit from a Keller-Segel system with additional cross-diffusion by Jungel et al.  \cite{JLW2020}.
But to our knowledge, no derivation as a mean-field
limit of a microscopic cell-system has been performed for $d\geq 2$.

The main purpose of this paper is to derive the parabolic-parabolic KS system \eqref{diffaggKSeq} from  a stochastic particle system. Let $\big(\Omega, \mathcal{F},  \mathbb{P}\big)$ be a probability space endowed with $N$ independent
$d$-dimensional Brownian motions $\{{B_t^i}\}_{i=1}^N$. We will derive the system \eqref{diffaggKSeq} as a mean-field limit of the following stochastic particle system with regularized interaction
\begin{eqnarray}\label{interparODES}\nonumber
X_t^{i,\varepsilon}&=&X_0^i+\frac{1}{N}\sum\limits_{j=1}^N \int_0^{t}\int_0^{s-\varepsilon}  \frac{e^{-\frac{|X_s^{i,\varepsilon}-X_r^{j,\varepsilon}|^2}{4 (s-r)} +\lambda (r-s) } }   {\big(4\pi(s-r)\big)^{\frac{d}{2}}}  \frac{X_r^{j,\varepsilon}-X_s^{i,\varepsilon}}{2(s-r)}drds\\
&&+\int_0^{t}e^{-\lambda s}\int_{\mathbb{R}^{d}}\frac{e^{-\frac{|X_s^{i,\varepsilon}-y|^2}{4s} } }   {(4\pi s)^{\frac{d}{2}}}\nabla c_0(y)dyds+\sqrt{2}B_t^i, ~~\hspace{0.2cm} t\geq 0,\hspace{0.2cm} 1\leq i\leq N,
\end{eqnarray}
where the initial data $\{X^{i}_0\}^N_{i=1}$ are i.i.d. random variables with the common density $\rho_0$.  Here, the second term on the right hand integrates on the interval $(0, t-\varepsilon)$ of the system \eqref{interparODES} is defined by
\begin{eqnarray}
\int_0^{t-\varepsilon} g(s)ds=
\left\{\begin{array}{ll}
\int_0^{t-\varepsilon} g(s)ds & \varepsilon < t, \\
0 &0 \leq t \leq \varepsilon.
\end{array}\right.
\end{eqnarray}
In this paper we assume that the initial data $(\rho_0, c_0)$ satisfy
\begin{eqnarray}\label{inicon1}
& \|\rho_0\|_{L^\infty(\mathbb{R}^d)}<+\infty, \int_{\mathbb{R}^{d}}|x|\rho_0 dx<+\infty, \|c_0\|_{W^{2,3}(\mathbb{R}^d)}<+\infty,  \|c_0\|_{W^{2,\infty}(\mathbb{R}^d)}<+\infty;  &\\ \label{inicon3}
&\|\nabla\rho_0\|_{L^{p}(\mathbb{R}^{d})}+\|\rho_0\|_{C^{2+\alpha}(\mathbb{R}^{d})}+ \|c_0\|_{C^{2+\alpha}(\mathbb{R}^{d})}<+\infty \mbox{~for~some~} \alpha>0,  \forall~p\in[1, d+1),&\\ \label{inicon2}
&\left\{\begin{array}{ll}
M_0=\| \rho_0(\cdot) \|_{L^1}\leq C, & d=2, \\
\|\rho_0\|_{L^{a}(\mathbb{R}^d)}+ \|\nabla c_0\|_{L^d(\mathbb{R}^d)}\leq C \mbox{  for some } \frac{d}{2}<  a \leq d, & d\geq 3,
\end{array}\right.&
\end{eqnarray}
where C are some given constants and $M_0$ is given by one well-known important property of the parabolic-parabolic KS system \eqref{diffaggKSeq} that the total mass is conserved
$$\|\rho(\cdot,t)\|_{L^1} =\| \rho_0(\cdot) \|_{L^1} = M_0.$$

In the present work, we couple the parameters $N$ and $\varepsilon$ by taking a sequence $\varepsilon\rightarrow 0$ as $N\rightarrow\infty$. Our main result is stated by the following theorem.
\begin{thm}\label{meanfieldres} ({\bf propagation of chaos})
Suppose $\{{X}^i_0\}_{i=1}^N$ are i.i.d. random variables with the common density  $\rho_0$ and $(\rho_0, c_0)$ satisfy the conditions \eqref{inicon1}, \eqref{inicon2}, \eqref{inicon3}. For any fixed $T>0$, let $\{({X}_t^{i,\varepsilon})_{t\leq T}\}^N_{i=1}$ be the unique strong solution to \eqref{interparODES}  with the i.i.d. initial data $\{{X}^{i} _0\}^N_{i=1}$ and Brownian motions $\{{B}^{i} _t\}^N_{i=1}$. Denote by $F_t(x_1,\cdots,x_N)$ the joint marginal  distribution of $\big({X}^{1,\varepsilon}_t,\cdots,{X}^{N,\varepsilon}_t\big)_{t\leq T}$, $f_t^{(j),\varepsilon}=\int_{\mathbb{R}^{(N-j)d}}F_t(\cdot,dx_{j+1},\cdots,dx_N)$ be the $j$-th marginal distribution of $F_t(x_1,\cdots,x_N)$, and $f_t^{(j),\varepsilon}$ has a density function $\rho_t^{(j),\varepsilon}$. Then for any $j\geq1$, $T>0$, for cut-off parameters $\varepsilon(N)\sim (\ln N)^{-\frac{2}{d+2}}\rightarrow0$, it holds that as $N\rightarrow\infty$
\begin{eqnarray}\label{mean}
\int_{\mathbb{R}^{jd}}\varphi(x)\rho_t^{(j),\varepsilon}dx\rightarrow \int_{\mathbb{R}^d}\varphi(x)\rho_t^{\otimes j}dx, \quad \forall~\varphi(x)\in C_b(\mathbb{R}^{jd}), ~ 0\leq t\leq T,
\end{eqnarray}
where $\rho$  is the unique weak solution to \eqref{diffaggKSeq} with the initial data $(\rho_0, c_0)$ and $c=e^{-\lambda t} e^{t\Delta} c_0+ \int_0^{t} e^{\lambda(s-t)} e^{(t-s)\Delta}\rho^{(1)}_s ds$.
\end{thm}
\begin{rem} The proof of Theorem \ref{meanfieldres} required the existence and uniqueness of the global weak solution to the KS system \eqref{diffaggKSeq}. Here the conditions \eqref{inicon1} and \eqref{inicon2} are sufficient for the existence of global weak solution to \eqref{diffaggKSeq} while the condition \eqref{inicon3} is to ensure that the weak solution is a classical solution and unique. Also, the $\epsilon$ modification of the drift term in the system \eqref{interparODES} plays a crucial role in terms of a threshold to get the global existence and uniqueness of strong solution to \eqref{interparODES}, see Theorem \ref{S}. Also, the logarithmic cut-off on $N$ of this paper in deriving propagation of chaos theory is technical in our proof. The case without a logarithmic cut-off on $N$ is very complex, which will be discussed in the future.\end{rem}

The proof of Theorem \ref{meanfieldres} is of length. Let us outline main steps here.  To study the convergence of the microscopic model to the macroscopic equation, we will split the proof into 3 steps. First, using the coupling method, as $N\rightarrow \infty$,  we show that the $N$ interacting particles $\{({X}^{i,\varepsilon}_t)_{t\geq0}\}_{i=1}^N$ respectively can be approximated by the processes $\{(\bar{X}_t^{i,\varepsilon})_{t\geq0}\}_{i=1}^N$, which are i.i.d processes and satisfy the following regularized self-consistent SDEs:
 \begin{eqnarray}\nonumber \label{intesde}
\bar{X}_t^{i,\varepsilon} &=&X^i_0+ \int_0^t\int_0^{s-\varepsilon}\int_{\mathbb{R}^{d}} \frac{e^{-\frac{|\bar{X}_s^{i,\varepsilon}-y|^2}{4 (s-r)} +\lambda (r-s) } }   {\big(4\pi(s-r)\big)^{\frac{d}{2}}}  \frac{y-\bar{X}_s^{i,\varepsilon}}{2(s-r)} df^{i,\varepsilon}_r(y) drds\\
&&+ \int_0^te^{-\lambda s}\int_{\mathbb{R}^{d}}\frac{e^{-\frac{|\bar{X}_s^{i,\varepsilon}-y|^2}{4s} } }   {(4\pi s)^{\frac{d}{2}}}\nabla c_0(y)dyds+\sqrt{2}B_t^i, ~~\hspace{0.2cm} t\geq 0,\hspace{0.2cm} 1\leq i \leq N,
\end{eqnarray}
 where $f_t^{i,\varepsilon}(x)$ is the time marginals of $(\bar{X}_t^{i,\varepsilon})_{t\geq0}$ and the initial data $\{X^{i}_0\}^N_{i=1}$ and Brownian motions $\{({B}_t^i)_{t\geq0}\}_{i=1}^N$ are the same as those of \eqref{interparODES}.  In Proposition \ref{P}, we derive that
  \begin{eqnarray}
\mathbb{E}\big[\sup\limits_{t\in[0,T]}|{X}_t^{i,\varepsilon}-\bar{{X}}_t^{i,\varepsilon}|\big]\leq\frac{C}{{\sqrt {N}\varepsilon^{\frac{d}{2}}}}\exp\big({\frac{C}{\varepsilon^{\frac{d}{2}+1}}}\big).
\end{eqnarray}

Second, we show that $f_t^{\varepsilon}$ has a density function $\rho^\varepsilon$ and $\rho^\varepsilon$ solves the following intermediate non-local system:
\begin{eqnarray}\label{pdeinter}
\left\{\begin{aligned}{}
&\partial_t\rho^\varepsilon-\nabla\cdot\big(\nabla \rho^\varepsilon  - \rho^\varepsilon\nabla c^\varepsilon \big)=0, 
\\
& c^\varepsilon=e^{-\lambda t} e^{t\Delta} c_0+ \int_0^{t-\varepsilon} e^{\lambda(s-t)} e^{(t-s)\Delta}\rho^\varepsilon(x,s) ds , 
\\
&\rho^\varepsilon(x,0)=\rho_0(x),  \hspace{0.5cm}  c^\varepsilon(x,0)=c_0(x), 
\end{aligned}\right.
\end{eqnarray}
where the heat semigroup operator $e^{t\Delta}$ is defined by
$$(e^{t\Delta }h)(x,t):=  \int_{\mathbb{R}^{d}}\frac{e^{-\frac{|x-y|^2}{4t} } }   {(4\pi t)^{\frac{d}{2}}}h(y)dy.$$
This is a parabolic-parabolic system with cross-diffusion, whose wellposedness theory is different from the standard parabolic system. But, here, we can obtain the uniform estimates of $(\rho^\varepsilon, c^\varepsilon)$ uniformly on $\epsilon$, see Proposition \ref{higeuniest}.
By compactness method, we prove that there exists a subsequence $(\rho^\varepsilon, c^\varepsilon)$ without relabeling goes to $(\rho, c)$ as $\varepsilon\rightarrow 0$, and $(\rho, c)$ is a weak solution to the diffusion-aggregation equation \eqref{diffaggKSeq}, see Theorem \ref{conthe}.
Also, we show that the weak solution has the following smoothness
\begin{eqnarray}
\|\rho\|_{C^{2,1}(\overline{\Omega}_T)}+\| c\|_{C^{2,1}(\overline{\Omega}_T)}\leq C,
\end{eqnarray}
and, then, based this regularity estimate and by constructing the corresponding SDE \eqref{stoKSfin} of \eqref{diffaggKSeq}, we prove that, for any $1\leq i\leq N$ and $T>0$,
\begin{eqnarray}
\mathbb{E}\big[\sup\limits_{t\in[0,T]}|\bar{{X}}^{i,\varepsilon}_t-{X}^i_t|\big]\leq  C_Te^{C_T T} \varepsilon,
\end{eqnarray}
where $\big\{({X}^{i}_t)_{t\geq0}\big\}^N_{i=1}$ is the unique strong solution to \eqref{stoKSfin} with the initial data $\{X^{i}_0\}^N_{i=1}$ and Brownian motions $\{({B}_t^i)_{t\geq0}\}_{i=1}^N$ being the same as those of \eqref{interparODES}. See Proposition \ref{fincoup}.

Third, we combine the first two steps together and couple the parameters $N$ and $\varepsilon$ by taking a sequence $\varepsilon\sim (\ln N)^{-\frac{2}{d+2}}\rightarrow0$ as $N\rightarrow\infty$. The joint limit shows that the trajectories of particle system \eqref{interparODES} converge to the trajectories of SDE \eqref{stoKSfin}, and then the propagation of chaos result \eqref{mean} is a corollary from this.

This paper is organized as follows. In section 2, we use the Picard successive approximation method to prove the well-posedness of the regularized $N$-particle system \eqref{interparODES}. In section 3, using the tool of Kantorovich-Rubinstein distance, we prove the well-posedness of the intermediate equation \eqref{intesde} and then estimate the
trajectories  between \eqref{interparODES} and \eqref{intesde}. In section 4,  we give some uniform estimates for the weak solution to the intermediate equation \eqref{pdeinter} and show that there exists a subsequence goes to a weak solution of \eqref{diffaggKSeq}. In section 5,  we prove the propagation of chaos results Theorem \ref{meanfieldres}. Finally, in the Appendix, we provide a supplementary proof of Lemma \ref{energeyest},  Proposition \ref{uniest} and  Theorem \ref{conthe}.


\section{Well-posedness for the regularized $N$-particle system}
This section is devoted to solve the stochastic differential equations \eqref{interparODES}. Since we havn't find a concrete result from SDE theory that can be cited directly, for completeness, we prove it by the standard Picard successive approximation method. For simplicity, we omit the superscript $\varepsilon$.
\begin{thm}\label{S}
For any fixed $\varepsilon>0$, $T>0$, assume the initial data $\nabla c_0\in W^{1,\infty}$, then the  stochastic differential equations \eqref{interparODES} has a unique (t-continuous) strong solution $\{X^i_t(\omega)\}_{i=1}^N$.
\end{thm}
\begin{proof}
{\bf Step 1.} The proof of the existence is similar to the existence proof for ordinary differential equations by Picard successive approximation method. We construct a approximating solution sequence in this step.

Define ${\bf{Y}}_t^{(0)}=\{X^1_0(\omega),X^2_0(\omega),\cdots,X^N_0(\omega)\}$ and ${\bf{Y}}_t^{(k)}=\{Y_t^{1,(k)}(\omega),Y_t^{2,(k)}(\omega),\cdots,Y_t^{N,(k)}(\omega)\}$ inductively as follows
\begin{eqnarray}\label{SDEinduc}\nonumber
Y_t^{i,(k+1)}&=&X^i_0+\frac{1}{N}\sum\limits_{j=1}^N\int_0^{t}\int_0^{s-\varepsilon} \frac{e^{-\frac{|Y_s^{i,(k)}-Y_r^{j,(k)}|^2}{4 (s-r)} +\lambda (r-s) } }   {\big(4\pi(s-r)\big)^{\frac{d}{2}}} \frac{Y_r^{j,(k)}-Y_s^{i,(k)}}{2(s-r)}drds\\
&&+\int_0^{t}e^{-\lambda s}\int_{\mathbb{R}^{d}}\frac{e^{-\frac{|Y_s^{i,(k)}-y|^2}{4s} } }   {(4\pi s)^{\frac{d}{2}}}\nabla c_0(y)dyds+\sqrt{2}{B}^i_t, ~~\hspace{0.1cm} ~1\leq i \leq N,
\end{eqnarray}
for $k\geq 0$, $t\leq T$.
Notice that
\begin{eqnarray}\label{1ststep}\nonumber
\mathbb{E}[|Y_t^{i,(1)}-X^i_0|^2]&=& \mathbb{E}\big[\big(\frac{1}{N}\sum\limits_{j=1}^N\int_0^{t}\int_0^{s-\varepsilon} \frac{e^{-\frac{|X_0^i-X_0^j|^2}{4 (s-r)} +\lambda (r-s) } }   {\big(4\pi(s-r)\big)^{\frac{d}{2}}} \frac{X_0^j-X_0^i}{2(s-r)}drds\\ \nonumber
&&+\int_0^{t}e^{-\lambda s}\int_{\mathbb{R}^{d}}\frac{e^{-\frac{|X_0^i-y|^2}{4s} } }   {(4\pi s)^{\frac{d}{2}}}\nabla c_0(y)dyds+\sqrt{2}{B}_t\big)^2\big]\\  \nonumber
&\leq& 3\mathbb{E}\big[\big(\frac{1}{N}\sum\limits_{j=1}^N\int_0^{t}\int_0^{s-\varepsilon} \frac{e^{-\frac{|X_0^i-X_0^j|^2}{4 (s-r)} +\lambda (r-s) } }   {\big(4\pi(s-r)\big)^{\frac{d}{2}}}  \frac{|X_0^j-X_0^i|}{2(s-r)}drds\big)^2 \big]
\\  \nonumber
&&+ 3\mathbb{E}\big[\|\nabla c_0\|^2_{\infty}\big(\int_0^{t}\int_{\mathbb{R}^{d}}\frac{e^{-\frac{|X^i_0-y|^2}{4s} } }   {(4\pi s)^{\frac{d}{2}}}dyds\big)^2\big]+6\mathbb{E}\big[|{B}_t|^2\big]\\
&\leq& C t,    \quad  1\leq i\leq N,
\end{eqnarray}
where the constant $C$ only depends on $\varepsilon$, $d$, $T$ and $\|\nabla c_0\|_{\infty}$. Notice that in the last inequality, we used the following inequality
\begin{eqnarray}\label{maxim1}
 g(u)=u e^{-u^2}\leq \frac{\sqrt{2}}{2}e^{-\frac{1}{2}}, \hspace {1cm} u \geq0.
\end{eqnarray}

Define
$$
a^{i,(k)}(s,\omega)=\frac{1}{N}\sum\limits_{j=1}^N\int_0^{s-\varepsilon} \frac{e^{-\frac{|Y_s^{i,(k)}-Y_r^{j,(k)}|^2}{4 (s-r)} +\lambda (r-s) } }   {\big(4\pi(s-r)\big)^{\frac{d}{2}}}  \frac{Y_r^{j,(k)}-Y_s^{i,(k)}}{2(s-r)}dr+e^{-\lambda s}\int_{\mathbb{R}^{d}}\frac{e^{-\frac{|Y_s^{i,(k)}-y|^2}{4s} } }   {(4\pi s)^{\frac{d}{2}}}\nabla c_0(y)dy
$$
for $k\geq 0$, $s\leq T$ and $1\leq i \leq N$.  A simple computation shows that for all $k\geq 1$,
\begin{eqnarray}\label{Lips}\nonumber
&&|a^{i,(k)}(s,\omega)-a^{i,(k-1)}(s,\omega)|\\ \nonumber
&\leq& \frac{C}{N}\sum\limits_{j=1}^N\int_0^{s-\varepsilon} \frac{|Y_s^{i,(k)}-Y_s^{i,(k-1)}|+|Y_r^{j,(k)}-Y_r^{j,(k-1)}|}{(s-r)^{\frac{d}{2}+1}}dr\\ \nonumber
&&+\int_{\mathbb{R}^{d}}\frac{e^{-\frac{|y|^2}{4s} } }   {(4\pi s)^{\frac{d}{2}}}\int_0^1\big|D^2 c_0\big(\theta (Y_s^{i,(k)}-y)+(1-\theta)(Y_s^{i,(k-1)}-y)\big)\big|d\theta dy |Y_s^{i,(k)}-Y_s^{i,(k-1)}|\\ \nonumber
&\leq& \frac{C}{N\varepsilon^{\frac{d}{2}+1}}\sum\limits_{j=1}^N\int_0^{s-\varepsilon} |Y_s^{i,(k)}-Y_s^{i,(k-1)}|+|Y_r^{j,(k)}-Y_r^{j,(k-1)}|dr\\ \nonumber
&&+\|D^2c_0\|_\infty\int_{\mathbb{R}^{d}}\frac{e^{-\frac{|y|^2}{4s} } }   {(4\pi s)^{\frac{d}{2}}}dy|Y_s^{i,(k)}-Y_s^{i,(k-1)}|\\ \nonumber
&\leq& C\big(\frac{1}{N}\sum\limits_{j=1}^N\int_0^{s-\varepsilon} |Y_r^{j,(k)}-Y_r^{j,(k-1)}|dr+|Y_s^{i,(k)}-Y_s^{i,(k-1)}| \big),
\end{eqnarray}
where $C$ only depends on $d$, $T$, $\varepsilon$ and $\|D^2c_0\|_\infty$.
Using the above inequality and H\"{o}lder's inequality, 
one has
\begin{eqnarray}\label{kststep}\nonumber
&&\mathbb{E}[|Y_t^{i,(k+1)}-Y_t^{i,(k)}|^2]\\ \nonumber
&=& \mathbb{E}\big[\big(\int_0^{t}(a^{i,(k)}(s,\omega)-a^{i,(k-1)}(s,\omega)) ds\big)^2\big]\\  \nonumber
&\leq& C\mathbb{E}\big[\big(\int_0^{t}\frac{1}{N}\sum\limits_{j=1}^N\int_0^{s-\varepsilon} |Y_r^{j,(k)}-Y_r^{j,(k-1)}|drds+\int_0^{t}|Y_s^{i,(k)}-Y_s^{i,(k-1)}| ds\big)^2\big]\\ \nonumber
&\leq& C\mathbb{E}\big[\big(\frac{1}{N}\sum\limits_{j=1}^N\int_0^{t}\int_0^{t} |Y_r^{j,(k)}-Y_r^{j,(k-1)}|drds+\int_0^{t}|Y_s^{i,(k)}-Y_s^{i,(k-1)}| ds\big)^2\big]\\
&\leq& \frac{C}{N}\sum\limits_{j=1}^N\int_0^{t}\mathbb{E}\big[|Y_s^{j,(k)}-Y_s^{j,(k-1)}|^2 \big]ds+C\int_0^{t}\mathbb{E}\big[|Y_s^{i,(k)}-Y_s^{i,(k-1)}|^2\big] ds,
\end{eqnarray}
where the constant  $C$ only depends on $d$, $T$, $\varepsilon$ and $\|D^2c_0\|_\infty$.

Combining \eqref{1ststep}, \eqref{kststep}  and by induction on $k$, we obtain
\begin{eqnarray}\nonumber
\max\limits_{1\leq i \leq N}\mathbb{E}[|Y_t^{i,(2)}-Y_t^{i,(1)}|^2]
&\leq& \frac{C}{N}\sum\limits_{j=1}^N\int_0^{t}\max\limits_{1\leq j \leq N}\mathbb{E}\big[|Y_s^{j,(1)}-Y_s^{j,(0)}|^2 \big]ds\\ \nonumber
&&+C\int_0^{t}\max\limits_{1\leq i \leq N}\mathbb{E}\big[|Y_s^{i,(1)}-Y_s^{i,(0)}|^2\big] ds\\
&\leq&\frac{C^{2}t^{2}}{2!} ,
\end{eqnarray}
$$
\vdots
$$
\begin{eqnarray}\label{inducststep}
\max\limits_{1\leq i \leq N}\mathbb{E}[|Y_t^{i,(k+1)}-Y_t^{i,(k)}|^2]
\leq \frac{C^{k+1}t^{k+1}}{(k+1)!} ,
\end{eqnarray}
for some suitable constant $C$ only depends on $d$, $T$, $\varepsilon$, $\|\nabla c_0(y)\|_{\infty}$ and $\|D^2c_0\|_\infty$.

Now a similar computation with \eqref{kststep} and using \eqref{inducststep} show that
\begin{eqnarray} \nonumber
\mathbb{E}[\sup\limits_{0\leq t \leq T}|Y_t^{i,(k+1)}-Y_t^{i,(k)}|^2]&\leq&\frac{C}{N}\sum\limits_{j=1}^N\int_0^{T}\mathbb{E}\big[|Y_s^{j,(k)}-Y_s^{j,(k-1)}|^2 \big]ds\\ \nonumber
&&+C\int_0^{T}\mathbb{E}\big[|Y_s^{i,(k)}-Y_s^{i,(k-1)}|^2\big] ds\\
&\leq& \frac{C^{k+1}T^{k+1}}{(k+1)!} \mbox{~~for~any~} 1\leq i \leq N.
\end{eqnarray}

{\bf Step 2.} We get a strong solution of \eqref{interparODES} by taking limit $k\rightarrow\infty$.

Since, for any $1\leq i \leq N$, using Markov's inequality, one has
\begin{eqnarray}
\mathbb{P}[\sup\limits_{0\leq t \leq T}|Y_t^{i,(k+1)}-Y_t^{i,(k)}|>\frac{1}{2^k}]\leq 2^{2k} \mathbb{E}[\sup\limits_{0\leq t \leq T}|Y_t^{i,(k+1)}-Y_t^{i,(k)}|^2]\leq 2^{2k}\frac{C^{k+1}T^{k+1}}{(k+1)!}
\end{eqnarray}
and
\begin{eqnarray}
\sum\limits_{k=1}^\infty 2^{2k}\frac{C^{k+1}T^{k+1}}{(k+1)!}<\infty.
\end{eqnarray}
The Borel-Cantelli Lemma thus applies
\begin{eqnarray}
\mathbb{P}[\sup\limits_{0\leq t \leq T}|Y_t^{i,(k+1)}-Y_t^{i,(k)}|>\frac{1}{2^k}~~\mbox{i.o.} ]=0,
\end{eqnarray}
where  i.o. is an abbreviation for infinitively often. In light of this,
\begin{eqnarray}\label{inducststepstrong}
Y_t^{i,(k)}=Y_t^0+ \sum\limits_{n=0}^{k-1}( Y_t^{i,(n+1)}-Y_t^{i,(n)}) \xrightarrow{k\rightarrow\infty} X_t^{i} \mbox{~~for~a.e.~} \omega, ~~1\leq i \leq N,
\end{eqnarray}
uniformly on $[0, T]$.

Hence, taking limit $k\rightarrow\infty$ in \eqref{SDEinduc}, one gets that $\{X_t^i\}_{i=1}^N$ satisfies equation \eqref{interparODES}, i.e. for all $t\in[0,T]$,
\begin{eqnarray}\nonumber
X^i_t&=&X^i_0+\frac{1}{N}\sum\limits_{j=1}^N\int_0^{t}\int_0^{s-\varepsilon} \frac{e^{-\frac{|X^i_s-X^j_r|^2}{4 (s-r)} +\lambda (r-s) } }   {\big(4\pi(s-r)\big)^{\frac{d}{2}}}\frac{X^j_r-X^i_s}{2(s-r)}drds\\
&&-\int_0^{t}e^{-\lambda s}\int_{\mathbb{R}^{d}}\frac{e^{-\frac{|X^i_s-y|^2}{4s} } }   {(4\pi s)^{\frac{d}{2}}}\nabla c_0(y)dyds+\sqrt{2}{B}^i_t, \mbox{~a.s.~~for~} ~1\leq i \leq N,
\end{eqnarray}
and there is a t-continuous version of the right hand side of the above equation. Hence $\{X^i_t\}_{i=1}^N$ can be chosen to be t-continuous.

{\bf Step 3.} Finally, we prove that the strong solution of \eqref{interparODES} is unique.

Let ${\bf{X}}_t(\omega)=\{X^1(t,\omega), X^2(t,\omega),\cdots, X^N(t,\omega)\}$ and ${\bf{Y}}_t(\omega)=\{Y^1(t,\omega), Y^2(t,\omega)$, $\cdots$, $Y^N(t,\omega)\}$ be two solutions of  \eqref{interparODES} with the same initial data ${\bf{X}}_0$. Define for $ 1\leq i\leq N$,
$$
a^i(s,{\bf{X}})=\frac{1}{N}\sum\limits_{j=1}^N\int_0^{s-\varepsilon} \frac{e^{-\frac{|X^i_s-X^j_r|^2}{4 (s-r)} +\lambda (r-s) } }   {\big(4\pi(s-r)\big)^{\frac{d}{2}}}  \frac{X^j_r-X^i_s}{2(s-r)}dr+e^{-\lambda s}\int_{\mathbb{R}^{d}}\frac{e^{-\frac{|X^i_s-y|^2}{4s} } }   {(4\pi s)^{\frac{d}{2}}}\nabla c_0(y)dy.
$$
  A similar computation with \eqref{Lips} shows that
\begin{eqnarray}\label{Lips11}
|a^i(s,{\bf{X}})-a^{i}(s,{\bf{Y}})|\leq C\big(\frac{1}{N}\sum\limits_{j=1}^N\int_0^{s-\varepsilon} |X^j_r-Y^j_r|dr+|X^i_s-Y^i_s|\big),
\end{eqnarray}
where $C$ only depends on $d$, $T$, $\varepsilon$ and $\|D^2c_0\|_\infty$.

Then by H\"{o}lder's inequality, one has
\begin{eqnarray}\label{unique}\nonumber
\max\limits_{1\leq i \leq N}\mathbb{E}[|X^i_t-Y^i_t|^2]&=&\max\limits_{1\leq i \leq N} \mathbb{E}\big[\big(\int_0^{t}(a^i(s,{\bf{X}})-a^{i}(s,{\bf{Y}})) ds\big)^2\big]\\  \nonumber
&\leq& C\max\limits_{1\leq i \leq N}\mathbb{E}\big[\big(\frac{1}{N}\sum\limits_{j=1}^N\int_0^{t}\int_0^{s-\varepsilon} |X^j_r-Y^j_r|drds+\int_0^{t}|X^i_s-Y^i_s| ds\big)^2\big]\\ \nonumber
&\leq& C\int_0^{t}\max\limits_{1\leq i \leq N}\mathbb{E}\big[|X^i_s-Y^i_s|^2\big] ds.
\end{eqnarray}
Therefore, by the Gronwall's inequality,
$$
\max\limits_{1\leq i \leq N}\mathbb{E}[|X^i_t-Y^i_t|^2]=0, ~~\forall~ t\geq 0.
$$
 Hence
\begin{eqnarray}
\mathbb{P}\big[|{\bf{X}}_t-{\bf{Y}}_t|=0 ~\mbox{~for~all~} t\in\mathbf{Q}\cap[0,T]\big]=1,
\end{eqnarray}
where $\mathbf{Q}$ denotes the rational numbers.

By continuity of $t\rightarrow|{\bf{X}}_t-{\bf{Y}}_t|$, it follows that
\begin{eqnarray}
\mathbb{P}\big[|{\bf{X}}_t-{\bf{Y}}_t|=0 ~\mbox{~for~all~} t\in[0,T]\big]=1,
\end{eqnarray}
i.e. the pathwise uniqueness is proved.
\end{proof}

\section{Well-posedness for the intermediate SDEs and convergence argument}
In this section, following the sprit of \cite{S1991},  we show that  each particle path of \eqref{interparODES} has a natural limit $\bar{X}_t^{i,\varepsilon}$ when $N$ goes to infinity.  $\{\bar{X}_t^{i,\varepsilon}\}_{i=1}^N$ are i.i.d. and each path satisfies the following regularized self-consistent SDEs:
 \begin{eqnarray}\label{sto1equ}\nonumber
\bar{X}_t^{i,\varepsilon} &=&X^i_0+ \int_0^t\int_0^{s-\varepsilon}\int_{\mathbb{R}^{d}} \frac{e^{-\frac{|\bar{X}_s^{i,\varepsilon}-y|^2}{4 (s-r)} +\lambda (r-s) } }   {\big(4\pi(s-r)\big)^{\frac{d}{2}}}  \frac{y-\bar{X}_s^{i,\varepsilon}}{2(s-r)} df^{i,\varepsilon}_r(y) drds\\
&&+ \int_0^te^{-\lambda s}\int_{\mathbb{R}^{d}}\frac{e^{-\frac{|\bar{X}_s^{i,\varepsilon}-y|^2}{4s} } }   {(4\pi s)^{\frac{d}{2}}}\nabla c_0(y)dyds+\sqrt{2}B_t^i, ~~\hspace{0.2cm} t\geq 0,\hspace{0.2cm} 1\leq i \leq N,
\end{eqnarray}
 where $f_t^{i,\varepsilon}(x)$ is the time marginals of $(\bar{X}_t^{i,\varepsilon})_{t\geq0}$, the initial data $\{X^{i}_0\}^N_{i=1}$ and Brownian motions $\{({B}_t^i)_{t\geq0}\}_{i=1}^N$ are the same as those of \eqref{interparODES}.

For any $f\in L^\infty\big(0,T; \mathcal{P}_1(\mathbb{R}^{d})\big) $,  define the drift term of equation \eqref{sto1equ} as
\begin{eqnarray}\label{defmathcalT}
B^\varepsilon_{[f]}(t,x)=\int_0^{t-\varepsilon} \int_{\mathbb{R}^{d}}\frac{e^{-\frac{|x-y|^2}{4 (t-s)} +\lambda (s-t) } }   {\big(4\pi(t-s)\big)^{\frac{d}{2}}} \frac{y-x}{2(t-s)}df_s(y)ds+e^{-\lambda t}\int_{\mathbb{R}^{d}}\frac{e^{-\frac{|x-y|^2}{4t} } }   {(4\pi t)^{\frac{d}{2}}}\nabla c_0(y)dy.
\end{eqnarray}

In Subsection 3.1, we introduce the useful tool of Wasserstein distance. In Subsection 3.2, we use the standard contraction argument and fixed point theorem to get the existence and uniqueness of strong solution to \eqref{sto1equ}. Finally, in Subsection 3.3,  we give a trajectorial estimate between \eqref{interparODES} and \eqref{sto1equ}.

\subsection{Preliminaries}
First we introduce a topology of the 1-Wasserstein space which will be used for  proving the well-posedness of \eqref{sto1equ}. Define the following space
$$
\mathcal{P}_1(\mathbb{R}^{d})=\big\{f|~f \mbox{ is a probability measure on } \mathbb{R}^{d} \mbox{ and } \int_{\mathbb{R}^{d}}|x| df(x)<+\infty\big\}.
$$
The so called Kantorovich-Rubinstein distance in $\mathcal{P}_1(\mathbb{R}^{d})$ is given as follows
$$
\mathcal {W}_1(f,~g)=\inf_{\pi\in\Lambda(f,~g)}\Big\{\int_{\mathbb{R}^{d}\times\mathbb{R}^{d}}|x-y|d\pi(x,y)\Big\},
$$
where $\Lambda(f,~g)$ is the set of joint probability measures on ${\mathbb{R}^{d}}\times{\mathbb{R}^{d}}$ with marginals $f$ and $g$.  If $f,~g$ have densities $\rho_1,~\rho_2$ respectively, we also denote the distance as $\mathcal {W}_1(\rho_1,~\rho_2)$. In \cite[Theorem 6.18]{V}, it has been proven that $\mathcal{P}_1(\mathbb{R}^{d})$ endowed with this distance is a complete metric space.  And by  \cite[Theorem 6.9]{V}, the following proposition holds.
\begin{prop}\label{property}
For a sequence of $\big\{f_k\big\}_{k=1}^{\infty}$ and $f$ in $\mathcal{P}_1(\mathbb{R}^{d})$, the convergence of $\big\{f_k\big\}_{k=1}^{\infty}$ to $f$ in the  1-Wasserstein distance implies the narrow convergence of $\big\{f_k\big\}_{k=1}^{\infty}$, i.e.
$$
\mathcal {W}_1(f_k,f)\xrightarrow{k\rightarrow\infty} 0  \Rightarrow  \int\varphi df_k(x)\xrightarrow{k\rightarrow\infty}\int\varphi df(x) \mbox{~~for~any~} \varphi\in C_b(\mathbb{R}^{d}),
$$
 where $C_b(\mathbb{R}^{d})$ is the space of continuous and bounded functions.
\end{prop}

In this paper we use the following time dependent space $L^\infty\big(0,T; \mathcal{P}_1(\mathbb{R}^{d})\big)$:
$$
\big\{f(t,x)| ~f(t,\cdot) \mbox{ is a probability measure on } \mathbb{R}^{d}\mbox{ for any time $t$ } \mbox{ and} \sup\limits_{t\in{[0,T]}}\int_{\mathbb{R}^d}|x| df(t,x)<+\infty\big\}.
  $$
  endowed with metric
$$
\mathcal {M}_T(f_t^1,f_t^2)=\sup\limits_{t\in{[0,T]}}\mathcal {W}_1(f_t^1,f_t^2).
$$
 And the following proposition is well known, c.f. \cite{BCC2}.
 \begin{prop}\label{timedep}
 $\big(L^\infty\big(0,T; \mathcal{P}_1(\mathbb{R}^{d})\big),  \mathcal {M}_T \big)$  is a complete metric space.
\end{prop}

\subsection{well-posedness of the regularized nonlinear SDEs }

Before giving the existence and uniqueness of strong solution to \eqref{sto1equ}, we give the following two lemmas for properties of the drift term $B^\varepsilon_{[f]}(t,x)$.

\begin{lem} \label{Hlp}
Assume that  $\nabla c_0\in W^{1,\infty}$, then for any $f\in L^\infty\big(0, T; \mathcal{P}(\mathbb{R}^{d})\big)$, the following hold
\begin{enumerate}
  \item $B^\varepsilon_{[f]}(t,x)$ is continuous in $[0,T]\times\mathbb{R}^{d}$;
   \item  $|B^\varepsilon_{[f]}(t,x)|\leq O(\frac{1}{{\varepsilon ^{\frac{d-1}{2}}}})$, for all $(t,x)\in [0,T]\times \mathbb{R}^{d}$;
  \item $B^\varepsilon_{[f]}(t,x)$ is Lipschitz with respect to x,
i.e.  there exists $L^\varepsilon=O(\frac{1}{\varepsilon^{\frac{d}{2}}})$, such that for $t\in[0,T], x_1, x_2 \in \mathbb{R}^{d}$ ,
$$
\big|B^\varepsilon_{[f]}(t,x_1)-B^\varepsilon_{[f]}(t,x_2)\big|\leq{L^\varepsilon\big|x_1-x_2\big|}.
$$
\end{enumerate}
\end{lem}
\begin{proof}
The statement (1) is trivial, we only need to prove (2) and (3).

Simple computation showes that
\begin{eqnarray*}
|B^\varepsilon_{[f]}(t,x)|&\leq&e^{-\lambda t}\big(\int_0^{t-\varepsilon} \int_{\mathbb{R}^{d}}\frac{e^{-\frac{|x-y|^2}{4 (t-s)} +\lambda s } }   {\big(4\pi(t-s)\big)^{\frac{d}{2}}} \frac{|y-x|}{2(t-s)}df_s(y)ds+\int_{\mathbb{R}^{d}}\frac{e^{-\frac{|x-y|^2}{4t} } }   {(4\pi t)^{\frac{d}{2}}}|\nabla c_0(y)|dy\big)\\
&\leq&\|\nabla c_0\|_{\infty}+\int_0^{t-\varepsilon} \frac{e^{-\lambda (t-s) } }   {\big(4\pi(t-s)\big)^{\frac{d}{2}} \sqrt{t-s} } \int_{\mathbb{R}^{d}}\frac{|x-y|}{2 \sqrt{t-s}} e^{-\frac{|x-y|^2}{4 (t-s)} } df_s(y)ds\\
&\leq&\|\nabla c_0\|_{\infty}+C_d\int_\varepsilon^{t} \frac{1}{u^{\frac{d+1}{2}}}du\\
&\leq&\|\nabla c_0\|_{\infty}+\frac{C_d}{{\varepsilon ^{\frac{d-1}{2}}}}=O(\frac{1}{{\varepsilon ^{\frac{d-1}{2}}}}),
\end{eqnarray*}
notice that we used inequality \eqref{maxim1}.

 By application of the mean value theorem  for integral, one also has
\begin{eqnarray*}
&&|B^\varepsilon_{[f]}(t,x_1)-B^\varepsilon_{[f]}(t,x_2)|\\
&&\leq|x_1-x_2|\Big(\int_{\mathbb{R}^{d}}\frac{e^{-\frac{|y|^2}{4t} } }   {(4\pi t)^{\frac{d}{2}}}\int_0^1\big|D^2 c_0(\theta(x_1-y)+(1-\theta)(x_2-y)\big)\big|d\theta dy+\\
&&
C_d\int_0^{t-\varepsilon} \int_{\mathbb{R}^{d}}\int_0^1\frac{|y-(\theta x_1+(1-\theta)x_2)|^2 e^{-\frac{|y-(\theta x_1+(1-\theta)x_2)|^2}{4 (t-s)} } }   {(t-s)^{\frac{d}{2}+2}}d\theta df_s(y)ds\\
&&+C_d\int_0^{t-\varepsilon} \int_{\mathbb{R}^{d}}\int_0^1\frac{e^{-\frac{|y-(\theta x_1+(1-\theta)x_2)|^2}{4 (t-s)}  } }   {(t-s)^{\frac{d}{2}+1}} d\theta df_s(y)ds\Big)
\\
&&\leq \Big(\|D^2c_0\|_\infty+\frac{C_d }{\varepsilon^{\frac{d}{2}}}\Big)|x_1-x_2|=O(\frac{1}{\varepsilon^{\frac{d}{2}}})|x_1-x_2|,
\end{eqnarray*}
notice that we used the following inequality
\begin{eqnarray}\label{maxim2}
g(u)=u e^{-u}\leq e^{-1}, \hspace {1cm} u \geq0.
\end{eqnarray}
\end{proof}

\begin{lem} \label{l5}
For any $f$, $g\in L^\infty\big(0, T;~\mathcal{P}_1(\mathbb{R}^{d})\big)$, the following inequality holds
\begin{eqnarray}\label{l5este}
\big\|B^\varepsilon_{[f]}(t,\cdot)-B^\varepsilon_{[g]}(t,\cdot)\big\|_{L^{\infty}(\mathbb{R}^{d})}\leq O(\frac{1}{\varepsilon^{\frac{d}{2}+1}}) \int_0^{t-\varepsilon}\mathcal{W}_1(f_s,g_s) ds, \mbox{~for~any~} 0< t\leq T.
\end{eqnarray}
\end{lem}
\begin{proof}
For any $x\in \mathbb{R}^{d}$, we have
\begin{eqnarray*}
&&|B^\varepsilon_{[f]}(t,x)-B^\varepsilon_{[g]}(t,x)|\\
&\leq&\big|\int_0^{t-\varepsilon} \int_{\mathbb{R}^{d}}\frac{e^{-\frac{|x-y|^2}{4 (t-s)} +\lambda s } }   {\big(4\pi(t-s)\big)^{\frac{d}{2}}}  \frac{y-x}{2(t-s)}df_s(y)ds-\int_0^{t-\varepsilon} \int_{\mathbb{R}^{d}}\frac{e^{-\frac{|x-y|^2}{4 (t-s)} +\lambda s } }   {\big(4\pi(t-s)\big)^{\frac{d}{2}}} \frac{y-x}{2(t-s)}dg_s(y)ds\big|\\
&=:& I.
\end{eqnarray*}

Let $\pi_t$ be an optimal transportation plan between the measures $f_t$ and $g_t$. Then by the mean value theorem  for integral, we have the following estimate
\begin{eqnarray*}
I&=&\Big|\int_0^{t-\varepsilon} \int_{\mathbb{R}^{2d}}\frac{e^{-\frac{|x-y|^2}{4 (t-s)} +\lambda s } }   {\big(4\pi(t-s)\big)^{\frac{d}{2}}}  \frac{y-x}{2(t-s)}d\pi_s(y,z)ds\\
&&-\int_0^{t-\varepsilon} \int_{\mathbb{R}^{2d}}\frac{e^{-\frac{|x-z|^2}{4 (t-s)} +\lambda s } }   {\big(4\pi(t-s)\big)^{\frac{d}{2}}}  \frac{z-x}{2(t-s)}d\pi_s(y,z)ds\Big|\\
&\leq&C_d\Big(\int_0^{t-\varepsilon} \int_{\mathbb{R}^{2d}}\int_0^1\frac{|x-(\theta y+(1-\theta)z)|^2 e^{-\frac{|x-(\theta y+(1-\theta)z)|^2}{4 (t-s)} } }   {(t-s)^{\frac{d}{2}+2}}d\theta |y-z|d\pi_s(y,z)ds\\
&&+\int_0^{t-\varepsilon} \int_{\mathbb{R}^{2d}}\int_0^1\frac{e^{-\frac{|x-(\theta y+(1-\theta)z)|^2}{4 (t-s)}  } }   {(t-s)^{\frac{d}{2}+1}} d\theta|y-z|d\pi_s(y,z)ds\Big)\\
&\leq&C_d\Big(\int_0^{t-\varepsilon} \int_{\mathbb{R}^{2d}}\frac{|y-z|}   {(t-s)^{\frac{d}{2}+1}} d\pi_s(y,z)ds+\int_0^{t-\varepsilon} \int_{\mathbb{R}^{2d}}\frac{ |y-z| }   {(t-s)^{\frac{d}{2}+1}}d\pi_s(y,z)ds\Big)\\
&\leq&C_d\big( \frac{1}{\varepsilon^{\frac{d}{2}+1}} \big) \int_0^{t-\varepsilon}\mathcal{W}_1(f_s,g_s) ds=O(\frac{1}{\varepsilon^{\frac{d}{2}+1}})\int_0^{t-\varepsilon}\mathcal{W}_1(f_s,g_s) ds,
\end{eqnarray*}
which finishes the proof of \eqref{l5este}.
\end{proof}

Next, we use the standard contraction argument and fixed point theorem to obtain the existence and uniqueness of strong solution to \eqref{sto1equ}.
\begin{thm}\label{existunim}
Assume that $\nabla c_0\in W^{1,\infty}$.
For any Brownian motion ${B} _t$ and initial datum $X_0$ satisfying $\mathbb{E}[|X_0|]< +\infty$,  there exists a unique global strong solution to \eqref{sto1equ}.
\end{thm}
\begin{proof}
{\bf Step 1.} Take $T>0$ and define a map $$\mathcal{S}: \big(L^\infty\big(0,T; \mathcal{P}_1(\mathbb{R}^{d})\big),  \mathcal {M}_T \big) \rightarrow \big(L^\infty\big(0,T; \mathcal{P}_1(\mathbb{R}^{d})\big),  \mathcal {M}_T \big)$$ as follows. For any $f\in L^\infty\big(0,T; \mathcal{P}_1(\mathbb{R}^{d})\big) $, define the following linear SDE:
\begin{eqnarray}\label{sto1equlineared}
\bar{X}_t^{\varepsilon} =X_0+ \int_0^t B^\varepsilon_{[f]}(s,\bar{X}_s^{\varepsilon})ds+\sqrt{2}B_t, ~~\hspace{0.2cm} 0\leq t \leq T.
\end{eqnarray}
By Lemma \ref{Hlp}, one knows that $B^\varepsilon_{[f]}(t,x)$ is bounded and Lipschitz with respect to $x$ for any  function $f\in L^\infty\big(0,T; \mathcal{P}_1(\mathbb{R}^{d})\big)$. Hence there exists a unique strong solution to \eqref{sto1equlineared}.

Let $g$ be the time marginals of $\bar{X}_t^\varepsilon$ and $\mathcal{S} f=: g $.  By Lemma \ref{Hlp}, one knows that $B^\varepsilon_{[f]}(s,x)$ is bounded by a constant $C$, therefore
 \begin{eqnarray}
\sup\limits_{t\in[0,T]} \mathbb{E}[|\bar{X}_t^{\varepsilon}|]\leq\mathbb{E}[|X_0|]+CT+\sqrt{2}\sup\limits_{t\in[0,T]} \mathbb{E}[|B_t|]< +\infty,
 \end{eqnarray}
 i.e. $g\in L^\infty\big(0,T; \mathcal{P}_1(\mathbb{R}^{d})\big)$ and $\mathcal{S}$ is well defined.

{\bf Step 2.} We prove that $\mathcal{S}$ is a contraction for small $T>0$.

For any two functions $f^1, f^2\in L^\infty\big(0,T; \mathcal{P}_1(\mathbb{R}^{d})\big)$, solve the following  two linear SDEs:
\begin{eqnarray}
\bar{X}_t^{i,\varepsilon} =X_0+ \int_0^t B^\varepsilon_{[f^i]}(s,\bar{X}_s^{i,\varepsilon})ds+\sqrt{2}B_t, ~~\hspace{0.1cm} 0\leq t \leq T,~\hspace{0.1cm} i=1, 2.
\end{eqnarray}
Then one has
\begin{eqnarray}\label{fixedp}
\sup\limits_{s\leq t}|\bar{X}_s^{1,\varepsilon}-\bar{X}_s^{2,\varepsilon}| \leq \int_0^t \big|B^\varepsilon_{[f^1]}(s,\bar{X}_s^{1,\varepsilon}) -B^\varepsilon_{[f^2]}(s,\bar{X}_s^{2,\varepsilon})\big|ds.
\end{eqnarray}
By Lemma \ref{Hlp} and Lemma \ref{l5}, one also has
\begin{eqnarray}\label{fixedp11} \nonumber
|B^\varepsilon_{[f^1]}(s,\bar{X}_s^{1,\varepsilon}) -B^\varepsilon_{[f^2]}(s,\bar{X}_s^{2,\varepsilon})\big|&\leq&|B^\varepsilon_{[f^1]}(s,\bar{X}_s^{1,\varepsilon}) -B^\varepsilon_{[f^1]}(s,\bar{X}_s^{2,\varepsilon})\big|\\ \nonumber
&&+|B^\varepsilon_{[f^1]}(s,\bar{X}_s^{2,\varepsilon}) -B^\varepsilon_{[f^2]}(s,\bar{X}_s^{2,\varepsilon})\big|\\
&\leq& L^\varepsilon\big(|\bar{X}_s^{1,\varepsilon}-\bar{X}_s^{2,\varepsilon}|+ \int_0^{s-\varepsilon}\mathcal{W}_1(f^1_\tau,f^2_\tau) d\tau\big).
\end{eqnarray}
Plugging \eqref{fixedp11} into \eqref{fixedp} and using Gronwall's lemma, we have
\begin{eqnarray}\label{fixedp22}
\sup\limits_{s\leq t}|\bar{X}_s^{1,\varepsilon}-\bar{X}_s^{2,\varepsilon}| \leq L^\varepsilon e^{L^\varepsilon t} \int_0^t \int_0^{s-\varepsilon}\mathcal{W}_1(f^1_\tau,f^2_\tau) d\tau ds.
\end{eqnarray}
Then taking expectation of \eqref{fixedp22}, one has
\begin{eqnarray}\label{fixedp33}
 \mathcal {M}_t (\mathcal{S}(f^1), \mathcal{S}(f^2))
 \leq L^\varepsilon t^2 e^{L^\varepsilon t}  \mathcal {M}_t(f^1,f^2).
\end{eqnarray}
Now we can choose $T>0$ such that $\mathcal{S}$ is a contraction on $[0,T]$.

{\bf Step 3.} A priori estimates and global well-posedness.

From Step 2, one can obtain a unique fixed point of $\mathcal{S}$ within time interval $[0,T]$. We denote it by $f\in L^\infty\big(0,T; \mathcal{P}_1(\mathbb{R}^{d})\big)$. Then $\bar{X}_t^{\varepsilon}$ defined by \eqref{sto1equlineared} with this $f$ is a local solution to \eqref{sto1equ}. In Step 1, since $f_T\in \mathcal{P}_1(\mathbb{R}^d)$,  one can take $T$  as a new initial time and repeat the previous process, by notice that the constant $L^\varepsilon$ in \eqref{fixedp33} is uniform in time, to show that the model \eqref{sto1equ} has a unique weak solution in $t\in [T, 2T]$. One can continue this process and obtain a unique global solution in $[0, T^*)$, until the first moment of $f\in L^\infty\big([0,T^*);\mathcal{P}_1(\mathbb{R}^d)\big)$ goes to $\infty$ when $t\rightarrow T^*-0$.

Next, we will show that the blow up of the first moment won't  happen in finite time,   which means that $T^*=+\infty$. If not, we suppose that there exists a finite $T^*$ such that
\begin{eqnarray}\label{contra}
\lim\limits_{t\uparrow  T^*}\int_{\mathbb{R}^{d}}|x|df_t(x)=\infty.
\end{eqnarray}
Let $\bar{X}_t^{\varepsilon}$ satisfies
\begin{eqnarray}
\bar{X}_t^{\varepsilon} =X_0+ \int_0^tB^\varepsilon_{[f]}(s,\bar{X}_s^{\varepsilon})ds+\sqrt{2}B_t, ~~\hspace{0.2cm} t\geq 0,\hspace{0.2cm} 1\leq i \leq N,
\end{eqnarray}
where $f\in L^\infty\big(0,T; \mathcal{P}_1(\mathbb{R}^{d})\big)$ is it's time marginals in $0\leq t\leq T< T^*$. By Lemma \ref{Hlp},
$$
|B^\varepsilon_{[f]}(t,x)|\leq O(\frac{1}{{\varepsilon ^{\frac{d-1}{2}}}}), \mbox{~for all~} (t,x)\in [0,T]\times \mathbb{R}^{d}.
$$
Therefore,  one has the following uniform estimate for the first moment in time
\begin{eqnarray}
\int_{\mathbb{R}^{d}}|x|df_t(x)=\mathbb{E}[|\bar{X}_t^{\varepsilon}|]\leq \mathbb{E}[|X_0|]+O(\frac{1}{{\varepsilon ^{\frac{d-1}{2}}}})t+\sqrt{2t}\leq C(T^*)< \infty,  \forall~ t\in (0, T^*),
\end{eqnarray}
which contradicts \eqref{contra}.
\end{proof}

\subsection{A trajectorial estimate between \eqref{interparODES} and \eqref{sto1equ}}
\begin{prop}\label{P}
Suppose $\{({X}_t^{i,\varepsilon})_{t\geq0}\}^N_{i=1}$ and $\{(\bar{X}_t^{i,\varepsilon})_{t\geq0}\}^N_{i=1}$ are the unique strong solutions to \eqref{interparODES} and \eqref{sto1equ} respectively,  with the same i.i.d. initial data $\{{X}^{i} _0\}^N_{i=1}$ and Brownian motions $\{({B}^{i} _t)_{t\geq0}\}^N_{i=1}$. Then for any $\varepsilon>0, 1\leq i\leq N$ and $T>0$, one has
\begin{eqnarray}\label{reguesti}
\mathbb{E}\big[\sup\limits_{t\in[0,T]}|{X}_t^{i,\varepsilon}-\bar{{X}}_t^{i,\varepsilon}|\big]\leq\frac{C}{{\sqrt {N}\varepsilon^{\frac{d}{2}}}}\exp\big({\frac{C}{\varepsilon^{\frac{d}{2}+1}}}\big),
\end{eqnarray}
where $C$ is a constant depending only on $T$, $d$ and $\|D^2c_0\|_\infty$.
\end{prop}
\begin{proof}
For any path ${\bf{X}}_t=\{X^{1}_t, X^{2}_t,\cdots, X^{N}_t\}_{t\geq 0}$, define the following two functions:
\begin{eqnarray}
b^i(s,{\bf{X}}_s)=\frac{1}{N}\sum\limits_{j=1}^N\int_0^{s-\varepsilon} \frac{e^{-\frac{|X^i_s-X^j_r|^2}{4 (s-r)} +\lambda (r-s) } }   {\big(4\pi(s-r)\big)^{\frac{d}{2}}}  \frac{X^j_r-X^i_s}{2(s-r)}dr, \mbox{~for~} 1\leq i\leq N
\end{eqnarray}
and
\begin{eqnarray}
\bar {b}^i(s,{\bf{X}}_s)=\int_0^{s-\varepsilon}\int_{\mathbb{R}^{d}} \frac{e^{-\frac{|{X}_s^{i}-y|^2}{4 (s-r)} +\lambda (r-s) } }   {\big(4\pi(s-r)\big)^{\frac{d}{2}}} \frac{y-{X}_s^{i}}{2(s-r)}df^i_r(y)dr,  \mbox{~for~} 1\leq i\leq N,
\end{eqnarray}
 where $f^i_t(x)$ is the time marginals  of $({X}_t^{i})_{t\geq0}$.

Define
${\bf{X}}_t^\varepsilon:=\{X^{1,\varepsilon}_t, X^{2,\varepsilon}_t,\cdots, X^{N,\varepsilon}_t\}$ and ${\bf\bar{X}}^\varepsilon_t:=\{\bar{X}^{1,\varepsilon}_t, \bar{X}^{2,\varepsilon}_t,\cdots, \bar{X}^{N,\varepsilon}_t\}$.
Following the spirit of \cite{S1991} and then comparing the paths between \eqref{interparODES} and \eqref{sto1equ}, one has
\begin{eqnarray}\label{zongest1} \nonumber
&&\big|{X}_t^{i,\varepsilon}-\bar{{X}}_t^{i,\varepsilon}\big|\\
&\leq&\int_0^t\big| b^i(s,{\bf{X}}_s^\varepsilon)-\bar {b}^i(s,{\bf\bar{X}}^\varepsilon_s)\big|ds \\ \nonumber
&&+\big|\int_0^te^{-\lambda s}\int_{\mathbb{R}^{d}}\frac{e^{-\frac{|{X}_s^{i,\varepsilon}-y|^2}{4s} } }   {(4\pi s)^{\frac{d}{2}}}\nabla c_0(y)dyds-\int_0^te^{-\lambda s}\int_{\mathbb{R}^{d}}\frac{e^{-\frac{|\bar{X}_s^{i,\varepsilon}-y|^2}{4s} } }   {(4\pi s)^{\frac{d}{2}}}\nabla c_0(y)dyds\big|\\ \nonumber
&\leq&\int_0^t\big| b^i(s,{\bf{X}}_s^\varepsilon)-{b}^i(s,{\bf\bar{X}}^\varepsilon_s)\big|ds+
\int_0^t\big| b^i(s,{\bf\bar {X}}_s^\varepsilon)-\bar {b}^i(s,{\bf\bar{X}}^\varepsilon_s)\big|ds \\
&&+\|D^2c_0\|_\infty\int_0^te^{-\lambda s}\int_{\mathbb{R}^{d}}\frac{e^{-\frac{|y|^2}{4s} } }   {(4\pi s)^{\frac{d}{2}}}dy |{X}_s^{i,\varepsilon}-\bar{X}_s^{i,\varepsilon}|ds.
\end{eqnarray}
A simple computation  shows that
  \begin{eqnarray}\label{bLips}\nonumber
\big| b^i(s,{\bf{X}}_s^\varepsilon)-{b}^i(s,{\bf\bar{X}}^\varepsilon_s)\big|&\leq& \frac{C}{N}\sum\limits_{j=1}^N\int_0^{s-\varepsilon} \frac{|X_s^{i,\varepsilon}-\bar{X}_s^{i,\varepsilon}|+|X_r^{j,\varepsilon}-\bar{X}_r^{j,\varepsilon}|}{(s-r)^{\frac{d}{2}+1}}dr\\
&\leq& \frac{C}{N\varepsilon^{\frac{d}{2}+1}}\sum\limits_{j=1}^N\int_0^{s-\varepsilon} \big(|X_s^{i,\varepsilon}-\bar{X}_s^{i,\varepsilon}|+|X_r^{j,\varepsilon}-\bar{X}_r^{j,\varepsilon}|\big)dr,
\end{eqnarray}
where $C$ only depends on $d$.

Plugging \eqref{bLips} into \eqref{zongest1},  for $\varepsilon \ll 1$, one has
\begin{eqnarray}\label{zongpathest}\nonumber
\big|{X}_t^{i,\varepsilon}-\bar{{X}}_t^{i,\varepsilon}\big|
&\leq&\frac{C}{\varepsilon^{\frac{d}{2}+1}}\int_0^t\big(\frac{1}{N}\sum\limits_{j=1}^N\int_0^{s-\varepsilon} |X_r^{j,\varepsilon}-\bar{X}_r^{j,\varepsilon}|dr+|X_s^{i,\varepsilon}-\bar{X}_s^{i,\varepsilon}|\big)ds\\
&&+\int_0^t\big| b^i(s,{\bf\bar {X}}_s^\varepsilon)-\bar {b}^i(s,{\bf\bar{X}}^\varepsilon_s)\big|ds,
\end{eqnarray}
where $C$ only depends on $d$, $T$ and $\|D^2c_0\|_\infty$.

 From \eqref{zongpathest}, for any $t\in [0, T]$, one has
 \begin{eqnarray}\label{zongpathest11}\nonumber
\sup\limits_{s\in[0,t]}|{X}^{i,\varepsilon}_s-\bar{{X}}^{i,\varepsilon}_s|&\le&
\frac{C}{\varepsilon^{\frac{d}{2}+1}}\int_0^t\big(\frac{1}{N}\sum\limits_{j=1}^N\int_0^{s-\varepsilon} |X_r^{j,\varepsilon}-\bar{X}_r^{j,\varepsilon}|dr+|X_s^{i,\varepsilon}-\bar{X}_s^{i,\varepsilon}|\big)ds\\ \nonumber
&&+\int_0^t\big| b^i(s,{\bf\bar {X}}_s^\varepsilon)-\bar {b}^i(s,{\bf\bar{X}}^\varepsilon_s)\big|ds\\ \nonumber
&\le&
\frac{C}{\varepsilon^{\frac{d}{2}+1}}\int_0^t\big(\frac{1}{N}\sum\limits_{j=1}^N \sup\limits_{\tau\in[0,s]} |X_\tau^{j,\varepsilon}-\bar{X}_\tau^{j,\varepsilon}|+\sup\limits_{\tau\in[0,s]}|X_{\tau}^{i,\varepsilon}-\bar{X}_{\tau}^{i,\varepsilon}|\big)ds\\
&&+\int_0^t\big| b^i(s,{\bf\bar {X}}_s^\varepsilon)-\bar {b}^i(s,{\bf\bar{X}}^\varepsilon_s)\big|ds,
\end{eqnarray}
where $C$ only depends on $d$, $T$ and $\|D^2c_0\|_\infty$.

Denote by $m_{N+1}(\omega^1,\cdots,\omega^N, y)\in  \mathcal{P}(\mathcal{C}^{N+1}) $ ($\mathcal{C}=C([0,T],\mathbb{R}^d)$) the joint distribution of $\big({X}^{1,\varepsilon}_t,\cdots,{X}^{N,\varepsilon}_t, \bar{{X}}^{i,\varepsilon} _t \big)_{t\geq0}$ and $m_3(\omega^{i},\omega^{j},y)=\int_{\mathcal{C}^{N-2}}m_{N+1}(d\omega^1,\cdots,\omega^{i} ,\cdots,\omega^{j} ,\cdots,d\omega^N, y)$ for any $1\leq i\neq j\leq N$. Since $\{(\bar{{X}}^{i,\varepsilon} _t)_{t\geq0}\}_{i=1}^N$ are i.i.d. stochastic processes, which will be proved in Subsection 4.1, and $\{({X}^{i,\varepsilon}_t)_{t\geq0}\}^N_{i=1}$ are exchangeable stochastic processes,   then $m_3(\omega^{i},\omega^{j},y)=m_3(\omega^{j},\omega^{i},y)$ for any $1\leq i\neq j\leq N$. Furthermore,  we obtain the following exchangeability qualities: for any $t\in [0,T]$,
 \begin{eqnarray}\nonumber\label{exchange}
 &&\mathbb{E}[\sup\limits_{s\in[0,t]}|{X}^{i,\varepsilon}_s-\bar{{X}}^{i,\varepsilon} _s|]=\int_{\mathcal{C}^3}\sup\limits_{s\in[0,t]}|\omega^{i}_s-y_s| dm_3(\omega^{i},\omega^{j},y)\\ \nonumber
&=&\int_{\mathcal{C}^3}\sup\limits_{s\in[0,t]}|\omega^{i}_s-y_s| dm_3(\omega^{j},\omega^{i},y)
=\int_{\mathcal{C}^3}\sup\limits_{s\in[0,t]}|\omega^{j}_s-y_s| dm_3(\omega^{i},\omega^{j},y)\\
 &=&\mathbb{E}[\sup\limits_{s \in[0,t]}|{X}^{j,\varepsilon}_s-\bar{{X}}^{j,\varepsilon} _s|].
 \end{eqnarray}
  Hence taking expectation of \eqref{zongpathest11}, one has
\begin{eqnarray}
&&\mathbb{E}\big[\sup\limits_{s\in[0,t]}|{X}^{i,\varepsilon}_s-\bar{{X}}^{i,\varepsilon}_s|\big]\nonumber\\
&\leq&\frac{C}{\varepsilon^{\frac{d}{2}+1}}\int_0^t \mathbb{E}\big[\sup\limits_{\tau \in[0,s]}|{X}^{i,\varepsilon}_\tau-\bar{{X}}^{i,\varepsilon}_\tau|\big]ds+\int_0^t \mathbb{E}\big[\big| b^i(s,{\bf\bar {X}}_s^\varepsilon)-\bar {b}^i(s,{\bf\bar{X}}^\varepsilon_s)\big|\big]ds.
\end{eqnarray}
where $C$ only depends on $d$, $T$ and $\|D^2c_0\|_\infty$. Applying the Gronwall's Lemma, we have
\begin{eqnarray}\label{Stoest1}\nonumber
\mathbb{E}\big[\sup\limits_{s\in[0,t]}\big|{X}^{i,\varepsilon}_s-\bar{{X}}^{i,\varepsilon}_s\big|\big]
&\leq& \exp\big(\frac{Ct}{\varepsilon^{\frac{d}{2}+1}}\big)\int_0^t \mathbb{E}\big[\big| b^i(s,{\bf\bar {X}}_s^\varepsilon)-\bar {b}^i(s,{\bf\bar{X}}^\varepsilon_s)\big|\big]ds\\
&\leq& \exp\big(\frac{Ct}{\varepsilon^{\frac{d}{2}+1}}\big)\int_0^t \sqrt{\mathbb{E}\big[\big| b^i(s,{\bf\bar {X}}_s^\varepsilon)-\bar {b}^i(s,{\bf\bar{X}}^\varepsilon_s)\big|^2\big] }ds.
\end{eqnarray}
Since
\begin{eqnarray}\label{iidest}
\mathbb{E}\big[\big| b^i(s,{\bf\bar {X}}_s^\varepsilon)-\bar {b}^i(s,{\bf\bar{X}}^\varepsilon_s)\big|^2\big]&=&  \mathbb{E}\big[|\frac{1}{N}\sum\limits_{j=1 }^N A^i_j(s)|^2\big],
\end{eqnarray}
where
\begin{eqnarray}\nonumber
A_j^i(s)&:=&\int_0^{s-\varepsilon} \frac{e^{-\frac{|\bar{X}^{i,\varepsilon}_s-\bar{X}^{j,\varepsilon}_r|^2}{4 (s-r)} +\lambda (r-s) } }   {\big(4\pi(s-r)\big)^{\frac{d}{2}}}  \frac{\bar{X}^{j,\varepsilon}_r-\bar{X}^{i,\varepsilon}_s}{2(s-r)}dr\\
&&- \int_0^{s-\varepsilon}\int_{\mathbb{R}^{d}} \frac{e^{-\frac{|\bar{X}_s^{i,\varepsilon}-y|^2}{4 (s-r)} +\lambda (r-s) } }   {\big(4\pi(s-r)\big)^{\frac{d}{2}}}\frac{y-\bar{X}_s^{i,\varepsilon}}{2(s-r)}df^{i,\varepsilon}_r(y)dr.
\end{eqnarray}

Because $\{(\bar{{X}}^{i,\varepsilon} _t)_{t\geq0}\}_{i=1}^N$ are i.i.d. random variables,  when $j\neq k$, one has
\begin{eqnarray*}
\mathbb{E}\big[A_j^i(s)A_k^i(s)\big]=0.
\end{eqnarray*}
Hence
\begin{eqnarray}\label{Stoest2}
\mathbb{E}\big[\big| b^i(s,{\bf\bar {X}}_s^\varepsilon)-\bar {b}^i(s,{\bf\bar{X}}^\varepsilon_s)\big|^2\big]=\frac{1}{N^2}\mathbb{E}\big[\sum\limits_{ \scriptstyle  j=1 }^N A_j^i(s)A_j^i(s)\big]\leq \frac{\mathbb{E}\big[(A_2^1(s))^2\big]}{N}.
\end{eqnarray}
Recalling that $\{(\bar{{X}}^{i,\varepsilon} _t)_{t\geq0}\}_{i=1}^N$ are i.i.d. random variables, one has
\begin{eqnarray}\label{Stoest3}\nonumber
\mathbb{E}\big[(A_2^1(s))^2\big]&=&\mathbb{E}\big[\big(\int_0^{s-\varepsilon} \frac{e^{-\frac{|\bar{X}^{1,\varepsilon}_s-\bar{X}^{2,\varepsilon}_r|^2}{4 (s-r)} +\lambda (r-s) } }   {\big(4\pi(s-r)\big)^{\frac{d}{2}}}  \frac{\bar{X}^{2,\varepsilon}_r-\bar{X}^{1,\varepsilon}_s}{2(s-r)}dr\\ \nonumber
&&- \int_0^{s-\varepsilon}\int_{\mathbb{R}^{d}} \frac{e^{-\frac{|\bar{X}_s^{1,\varepsilon}-y|^2}{4 (s-r)} +\lambda (r-s) } }   {\big(4\pi(s-r)\big)^{\frac{d}{2}}}  \frac{y-\bar{X}_s^{1,\varepsilon}}{2(s-r)}df^{1,\varepsilon}_r(y)dr\big)^2\big]\\ \nonumber
&\leq&2\mathbb{E}\big[\big(\int_0^{s-\varepsilon} \frac{e^{-\frac{|\bar{X}^{1,\varepsilon}_s-\bar{X}^{2,\varepsilon}_r|^2}{4 (s-r)} +\lambda (r-s) } }   {\big(4\pi(s-r)\big)^{\frac{d}{2}}}  \frac{\bar{X}^{2,\varepsilon}_r-\bar{X}^{1,\varepsilon}_s}{2(s-r)}dr\big)^2\\ \nonumber
&&+\big(  \int_0^{s-\varepsilon}\int_{\mathbb{R}^{d}} \frac{e^{-\frac{|\bar{X}_s^{1,\varepsilon}-y|^2}{4 (s-r)} +\lambda (r-s) } }   {\big(4\pi(s-r)\big)^{\frac{d}{2}}}  \frac{y-\bar{X}_s^{1,\varepsilon}}{2(s-r)}df^{1,\varepsilon}_r(y)dr\big)^2\big]\\ \nonumber
&=&4 \mathbb{E}\big[\big(  \int_0^{s-\varepsilon}\int_{\mathbb{R}^{d}} \frac{e^{-\frac{|\bar{X}_s^{1,\varepsilon}-y|^2}{4 (s-r)} +\lambda (r-s) } }   {\big(4\pi(s-r)\big)^{\frac{d}{2}}} \frac{y-\bar{X}_s^{1,\varepsilon}}{2(s-r)}df^{1,\varepsilon}_r(y)dr\big)^2\big]
\\ \nonumber
&\leq&  4s \int_0^{s-\varepsilon}\int_{\mathbb{R}^{2d}} \frac{e^{-\frac{|x-y|^2}{2 (s-r)}  } }   {\big(4\pi(s-r)\big)^{d}}  \frac{|y-x|^2}{4(s-r)^2}df^{1,\varepsilon}_r(y)df^{1,\varepsilon}_s(x)dr\\
&\leq&C \int_0^{s-\varepsilon}\int_{\mathbb{R}^{2d}} \frac{1}   {(s-r)^{d+1}} df^{1,\varepsilon}_r(y)df^{1,\varepsilon}_s(x)dr\leq \frac{C}{\varepsilon^d},
\end{eqnarray}
where $C$ only depends on $d$ and $T$. Notice that in the last inequality, we used the inequality \eqref{maxim2}.

Combining (\ref{Stoest1}), (\ref{Stoest2}) and (\ref{Stoest3}) together, one obtains \eqref{reguesti}.
\end{proof}

\section{Uniform estimates for the intermediate PDEs and compactness argument}
\subsection{Uniform estimates for the intermediate PDEs}
 For $\varphi\in C_b^2(\mathbb{R}^{d})$, applying It$\hat{\mathrm o}$'s formular for \eqref{sto1equ}, for any $1\leq i \leq N$, one has
\begin{eqnarray}
\varphi(\bar{X}_t^{i,\varepsilon})=\varphi(X^{i}_0)+\sqrt{2}\int_0^t \nabla \varphi(\bar{X}_s^{i,\varepsilon})\cdot dB_s^i+ \int_0^t \big(B^\varepsilon_{[f^{i,\varepsilon}]}(s,\bar{X}_s^{i,\varepsilon})\cdot \nabla \varphi(\bar{X}_s^{i,\varepsilon})+\triangle \varphi(\bar{X}_s^{i,\varepsilon})\big)ds.
\end{eqnarray}
Taking expectation of the above equation and noticing that  $f^{i,\varepsilon}$ admits a time marginal density denoted by ${\rho}^{i,\varepsilon}(t,x)$ (see \cite[Theorem 9.1.9]{SV1979}), we know that ${\rho}^{i,\varepsilon}(t,x)$ is a weak solution to  the following equation:
\begin{eqnarray}\label{intermeKSeq}
\left\{\begin{aligned}{}
&\partial_t\rho^\varepsilon-\nabla\cdot\big(\nabla \rho^\varepsilon  - \rho^\varepsilon\nabla c^\varepsilon \big)=0, 
\\
& c^\varepsilon=e^{-\lambda t} e^{t\Delta} c_0+ \int_0^{t-\varepsilon} e^{\lambda(s-t)} e^{(t-s)\Delta}\rho^\varepsilon(x,s) ds , 
\\
&\rho^\varepsilon(x,0)=\rho_0(x),  \hspace{0.5cm}  c^\varepsilon(x,0)=c_0(x). 
\end{aligned}\right.
\end{eqnarray}

It is easy to prove that for fixed $\varepsilon>0$, \eqref{intermeKSeq} has a unique global weak solution $\rho^\varepsilon$ in the class of $L^\infty\big(0,T; L^2(\mathbb{R}^{d})\big)\cap L^2\big(0,T; H^1(\mathbb{R}^{d})\big)$, and $\int_{\mathbb{R}^d}\rho^\varepsilon(t,x)dx\equiv 1$. In Proposition \ref{uniest}, we will give the uniform estimates of ${\rho}^{i,\varepsilon}(t,x)$
with the initial data satisfying the conditions \eqref{inicon1} and \eqref{inicon2}. Hence ${\rho}^{i,\varepsilon}={\rho}^{\varepsilon}$ for $i =1,\cdots, N$, which means that $\{(\bar{{X}}^{i,\varepsilon} _t)_{t\geq0}\}_{i=1}^N$ are i.i.d..

We focus on the uniform estimates of the weak solution to  \eqref{intermeKSeq} in this subsection. First, we start with the fundamental estimates of $c_\varepsilon$. Recalling that
$$
c^\varepsilon=e^{-\lambda t} e^{t\Delta} c_0+ \int_0^{t-\varepsilon} e^{\lambda(s-t)} e^{(t-s)\Delta}f(x,s) ds,
$$
 we directly cite the following two lemmas as a preparatory work.   The first lemma follows immediately from the Young's inequality for the convolution  \cite[pp. 99]{LL2001} and the well-known results on $L^p$ regularity of parabolic equation \cite[Theorem X.12]{B1983} or \cite[Chapter IV, Section 3]{LSU1968} (also see \cite{IY2012,SK2006,WLC2019}).

\begin{lem}\label{estiofpara}
Let $d\geq 1$, $1\leq q \leq p \leq \infty$, $\frac{1}{q}-\frac{1}{p}<\frac{1}{d}$. Suppose $c_0\in W^{1,p}(\mathbb{R}^d)$ and $f\in L^\infty(0,T; L^{q}(\mathbb{R}^d))$. Then for all $t>0$,
\begin{eqnarray}\label{estpara1}
\|c^\varepsilon(t) \|_{L^p(\mathbb{R}^d)}&\leq & \|c_0 \|_{L^p(\mathbb{R}^d)}+C\|f \|_{L^\infty(0,T; L^{q}(\mathbb{R}^d))};\\  \label{estpara2}
\|\nabla c^\varepsilon(t) \|_{L^p(\mathbb{R}^d)}&\leq & \|\nabla c_0 \|_{L^p(\mathbb{R}^d)}+C\|f \|_{L^\infty(0,T; L^{q}(\mathbb{R}^d))},
\end{eqnarray}
where $C$ are positive constants depending on $p, q$ and $d$.

Furthermore, let $p>1$, if $c_0\in W^{2,p}(\mathbb{R}^d)$ and $f\in L^p(\Omega_T)$, one has
\begin{eqnarray}\label{estpara3}
\|c^\varepsilon \|_{W^{2,1}_p(\Omega_T)}\leq  C\big( \|c_0 \|_{W^{2,p}(\mathbb{R}^d)}+\|f \|_{L^p(\Omega_T)}\big),
\end{eqnarray}
where $C$ is a positive constant depending on $T$, $p$ and $d$, $\Omega_T=(0,T)\times\mathbb{R}^d$, and ${W^{2,1}_p(\Omega_T})$ is the Sobolev space involving time defined by
$$
W^{2,1}_p(\Omega_T)=\{u | u\in L^p(\Omega_T), D^r_tD_x^su \in L^p(\Omega_T) \mbox{~for~any~} 2r+|s|\leq 2\}
$$
with the following norm
$$
\|u\|_{{W^{2,1}_p(\Omega_T)}}=\sum\limits_{ 0\leq2r+|s|\leq 2 } \| D^r_tD_x^su\|_{L^p(\Omega_T)}.
$$
\end{lem}
\begin{proof}
The inequalities \eqref{estpara1} and \eqref{estpara2} are direct results of Lemma 2.1 in \cite{IY2012}. We just need to prove \eqref{estpara3}. Set
\begin{eqnarray}
\phi^\varepsilon(t,x)=:\int_0^{t-\varepsilon} e^{\lambda(s-t)} e^{(t-s)\Delta}f(x,s) ds=
\left\{\begin{array}{ll}
\int_0^{t-\varepsilon} e^{\lambda(s-t)} e^{(t-s)\Delta}f(x,s)ds & \varepsilon < t, \\
0 &0 \leq t \leq \varepsilon.
\end{array}\right.
\end{eqnarray}
Then one can easily check that $\phi^\varepsilon$ satisfies the following initial value problem
\begin{eqnarray}\label{vriaequa}
\left\{\begin{aligned}{}
&\partial_t \phi^\varepsilon =\triangle \phi^\varepsilon-\lambda \phi^\varepsilon+e^{-\lambda\varepsilon} e^{\varepsilon\Delta}f(x,t-\varepsilon), && x \in \mathbb{R}^{d}, ~t>\varepsilon,\\
& \phi^\varepsilon(x,\varepsilon)=0,&& x \in \mathbb{R}^{d},
\end{aligned}\right.
\end{eqnarray}
where $e^{\varepsilon\Delta}f(x,t-\varepsilon)=\int_{\mathbb{R}^{d}}\frac{e^{-\frac{|x-y|^2}{4\varepsilon} } }   {(4\pi \varepsilon)^{\frac{d}{2}}}f(y,t-\varepsilon)dy$.

Then the following inequality is a particular consequence of  \cite[Theorem X.12]{B1983},
\begin{eqnarray}\label{lap1}\nonumber
\|\phi^\varepsilon \|_{W^{2,1}_p((\varepsilon,T)\times\mathbb{R}^d)}&\leq&  C\|e^{-\lambda\varepsilon} e^{\varepsilon\Delta}f (x,t-\varepsilon)\|_{L^p((\varepsilon,T)\times \mathbb{R}^d)}\\
&\leq&  C\|e^{\varepsilon\Delta}f (x,t-\varepsilon)\|_{L^p((\varepsilon,T)\times \mathbb{R}^d)},
\end{eqnarray}
where $C$ is a constant independent of $\varepsilon$.

By the Young's inequality for the convolution  \cite[pp. 99]{LL2001}, one knows that
\begin{eqnarray}\label{lap2}
\|e^{\varepsilon\Delta}f (\cdot,t-\varepsilon)\|_{ L^{p}(\mathbb{R}^d)}\leq C \|f (\cdot,t-\varepsilon)\|_{ L^{p}(\mathbb{R}^d)},  \quad \forall ~t>\varepsilon,
\end{eqnarray}
where $C$ is a constant  depending only on $d$ and $p$.

Combining \eqref{lap1} and \eqref{lap2} together, one has
\begin{eqnarray}\label{lap3}
\|\phi^\varepsilon \|_{W^{2,1}_p(\Omega_T)}=\|\phi^\varepsilon \|_{W^{2,1}_p((\varepsilon,T)\times\mathbb{R}^d)}\leq   C\|f\|_{L^p((0,T-\varepsilon)\times \mathbb{R}^d)}\leq   C\|f\|_{L^p(\Omega_T)}.
\end{eqnarray}

Since $c^\varepsilon=e^{-\lambda t} e^{t\Delta} c_0+\phi^\varepsilon$, one obtains \eqref{estpara3} immediately.
\end{proof}

The following lemma gives us a variant of Gagliardo-Nirenberg type inequality, which was proved by Sugiyama \cite[Lemma 2.4]{S2006}.
\begin{lem}\label{GNinequaliy}
Let $d\in\mathbb{N}$, $m\geq1$, $a>2$, $u\in L^{q_1}(\mathbb{R}^d)$ with $q_1\geq 1$ and $u^{\frac{r+m-1}{2}}\in H^1(\mathbb{R}^d)$
 with  $r>0$. If $q_1\in [1, r+m-1]$ and $q_2\in [\frac{r+m-1}{2},\frac{a(r+m-1)}{2}]$ satisfy
 \begin{eqnarray}\nonumber
\left\{\begin{array}{ll}
1 \leq q_1\leq q_2\leq \infty & d=1, \\
1 \leq q_1\leq q_2< \infty & d=2,\\
1 \leq q_1\leq q_2\leq \frac{d(r+m-1)}{d-2} & d\geq 3,
\end{array}\right.
\end{eqnarray}
then
 \begin{eqnarray}
 \|u\|_{L^{q_2}(\mathbb{R}^d)}\leq C ^{\frac{2}{r+m-1}} \|u\|^{1-\theta}_{L^{q_1}(\mathbb{R}^d)} \|\nabla u^{\frac{r+m-1}{2}}\|^{\frac{2\theta}{r+m-1}}_{L^{2}(\mathbb{R}^d)}
 \end{eqnarray}
 where
 \begin{eqnarray}\nonumber
\theta=\frac{r+m-1}{2}\big(\frac{1}{q_1}-\frac{1}{q_2}\big)\big(\frac{1}{d}-\frac{1}{2}+\frac{r+m-1}{2q_1}\big)^{-1}\\
C=\left\{\begin{array}{ll}
C(d,a) & \frac{r+m-1}{2}\leq q_1, \\
C_0 (d,a)^{\frac{1}{\beta}} & 1\leq q_1 <\frac{r+m-1}{2} ,
\end{array}\right.\\
\beta=\frac{q_2-\frac{r+m-1}{2}}{q_2-q_1}\big[\frac{2q_1}{r+m-1}+(1-\frac{2q_1}{r+m-1})\frac{2d}{d+2}\big].
\end{eqnarray}
 \end{lem}

Next, we start from the intermediate system \eqref{intermeKSeq} to  derive the $L^r$ estimate. The proof will be given in the Appendix.
\begin{lem}\label{energeyest}
Let $(\rho_\varepsilon, c_\varepsilon)$ be solutions to \eqref{intermeKSeq} with the initial data $(\rho_0, c_0)$.  Then for any $r\geq 2$, we have the following estimates in different space dimensions.
\begin{eqnarray}\label{disestfor1d}
\|\rho_\varepsilon\|^r_{L^r(\mathbb{R})}+\frac{2(r-1) }{r}\int_0^t\| \nabla (\rho_\varepsilon^{\frac{r}{2}})\|^2_{L^2(\mathbb{R})}ds
\leq \|\rho_0\|^r_{L^r(\mathbb{R})}+C_{r,1}t\big(\|c_0 \|_{W^{2,r+1}(\mathbb{R})}+ M_0^{2r+1}\big), \\ \label{disestfor2d}
\|\rho_\varepsilon\|^r_{L^r(\mathbb{R}^2)}+\big(\frac{4(r-1) }{r}-C_{r,2}M_0\big)\int_0^t\| \nabla (\rho_\varepsilon^{\frac{r}{2}})\|^2_{L^2(\mathbb{R}^2)}ds
\leq \|\rho_0\|^r_{L^r(\mathbb{R}^2)}+C_{r,2}t\|c_0 \|_{W^{2,r+1}(\mathbb{R}^2)}, \\  \label{disestforhighd}
\frac{d}{dt}\|\rho_\varepsilon\|^r_{L^{r}(\mathbb{R}^d)}\leq 2(r-1)\|\nabla \rho_\varepsilon^{\frac{r}{2}}\|^2_{L^2(\mathbb{R}^d)}\big(-\frac{2}{r}+C_d\|\nabla c_0\|_{L^d(\mathbb{R}^d)} +C_{d,q} \sup\limits_{0\leq s\leq t} \|\rho_\varepsilon(s)\|_{L^{q}(\mathbb{R}^d)}\big),
\end{eqnarray}
the last inequality holds true for any $d\geq 3$, and $M_0=\|\rho_0\|_{L^1(\mathbb{R}^d)}$, $\frac{d}{2}< q \leq d$.
\end{lem}

From the above Lemma \ref{energeyest} and the standard time derivative estimate, we have the following uniform estimates.
\begin{prop}\label{uniest}
Let $(\rho_\varepsilon, c_\varepsilon)$  with $0\leq t \leq T$ be solutions to \eqref{intermeKSeq} with the initial data $(\rho_0, c_0)$.  Under the conditions \eqref{inicon1} and \eqref{inicon2}, we have the following uniform estimates:
  \begin{eqnarray} \label{domainest}
\|\rho_\varepsilon\|_{L^\infty(0,T; L^2(\mathbb{R}^d))} + \int_0^T\| \nabla \rho_\varepsilon\|^2_{L^2(\mathbb{R}^d)}ds
\leq C,\\
\label{Linfinity}
 \| \rho_\varepsilon \|_{L^\infty(0,T; {L^1 (\mathbb{R}^{d})})}= M_0, \quad  \|\rho_\varepsilon\|_{L^\infty(0,T; {L^\infty (\mathbb{R}^{d})})}\leq C,\\
 \label{timeest}
\|\partial_t\rho_\varepsilon\|_{L^2(0, T; H^{-1}(\mathbb{R}^d))} \leq C,
    \end{eqnarray}
where $C$ are constants depending only on $d, T, M_0$, $\|\rho_0\|_{L^\infty(\mathbb{R}^d)}$ and $\|c_0\|_{W^{2,3}(\mathbb{R}^d)}$.

 Furthermore, if $\int_{\mathbb{R}^{d}}|x|\rho_0 dx<+\infty$, then
\begin{eqnarray}\label{1moment}
\int_{\mathbb{R}^{d}}|x|\rho_\varepsilon  dx\leq C, \quad \forall~0\leq t\leq T,
\end{eqnarray}
where $C$ is a constant depending only on $d, T, M_0$, $\|\rho_0\|_{L^\infty(\mathbb{R}^d)}$ and $\|c_0\|_{W^{2,\infty}(\mathbb{R}^d)}$.
\end{prop}

\begin{prop}\label{higeuniest}
Let $(\rho_\varepsilon, c_\varepsilon)$  with $0\leq t \leq T$ be solutions to \eqref{intermeKSeq} with the initial data $(\rho_0, c_0)$ satisfying \eqref{inicon1}, \eqref{inicon2}  and \eqref{inicon3}.  Then  we have the following uniform estimates:
\begin{eqnarray}
\|\rho_\varepsilon\|_{C^{2,1}(\overline{\Omega}_T)}+\| c_\varepsilon\|_{C^{2,1}(\overline{\Omega}_T)}\leq C,
\end{eqnarray}
 where $C$ is a constant independent of $\varepsilon$, $\overline{\Omega}_T=[0,T]\times\mathbb{R}^d$ and the norm defined by
\begin{eqnarray}\nonumber
 \|u\|_{C^{2,1}(\overline{\Omega}_T)}=\sum\limits_{ 0\leq2r+|s|\leq 2 } \| D^r_tD_x^su\|_{L^\infty(\overline{\Omega}_T)}.
 \end{eqnarray}
\end{prop}
\begin{proof}
 Since for any fixed $T>0$, it has been proved in Proposition \ref{uniest} that
 \begin{eqnarray}\label{rhoest}
 \| \rho_\varepsilon \|_{L^\infty(0,T; {L^1 (\mathbb{R}^{d})})}= M_0; \quad  \|\rho_\varepsilon\|_{L^\infty(0,T; {L^\infty (\mathbb{R}^{d})})}\leq C; \quad
  \int_0^T\| \nabla \rho_\varepsilon\|^2_{L^2(\mathbb{R}^d)}ds
\leq C,
\end{eqnarray}
then by Lemma \ref{estiofpara}, one has
 \begin{eqnarray}\label{Linfiniforc}
\|\nabla c^\varepsilon \|_{L^\infty(\Omega_T)}&\leq&  \|\nabla c_0 \|_{L^\infty(\mathbb{R}^d)}+C\|\rho_\varepsilon \|_{L^\infty(\Omega_T)}\leq C;\\ \nonumber
\|\triangle c^\varepsilon \|_{L^p(\Omega_T)}&\leq&  C\big( \|c_0 \|_{W^{2,p}(\mathbb{R}^d)}+\|\rho_\varepsilon \|_{L^p(\Omega_T)}\big)\\ \label{trinfiniforc}
&\leq& C\big( \|c_0 \|_{W^{2,p}(\mathbb{R}^d)}+M_0^{\frac{1}{p}}\|\rho_\varepsilon \|^{\frac{p-1}{p}}_{L^\infty(\Omega_T)}\big)\leq C,
\end{eqnarray}
for any $p>1$.

Recall the parabolic equation
\begin{eqnarray}\label{heatequ}
\partial_t\rho^\varepsilon-\triangle\rho^\varepsilon= -\nabla\rho^\varepsilon \cdot\nabla c^\varepsilon-\triangle c^\varepsilon\rho^\varepsilon=: Z^\varepsilon.
\end{eqnarray}
Using  \eqref{rhoest}, \eqref{Linfiniforc} and \eqref{trinfiniforc}, one has
\begin{eqnarray}\nonumber\label{regula1}
\|Z^\varepsilon\|_{L^2(\Omega_T)}&\leq& \|\nabla c^\varepsilon \|_{L^\infty(\Omega_T)} \| \nabla \rho_\varepsilon\|_{L^2(\Omega_T)}+\|\rho_\varepsilon\|_{L^6(\Omega_T)}\|\triangle c^\varepsilon \|_{L^3(\Omega_T)}\\
&\leq& C.
\end{eqnarray}

From \eqref{heatequ}, one has the fact
\begin{eqnarray}
\rho_\varepsilon(t)=\gamma_t\ast_{t,x} Z^\varepsilon+\gamma_t\ast_x \rho_0;\quad \nabla\rho_\varepsilon=\Gamma_t\ast_{t,x} Z^\varepsilon+\gamma_t\ast_{x} \nabla\rho_0,
\end{eqnarray}
where
\begin{eqnarray}
\gamma_t(x)=\frac{e^{-\frac{|x|^2}{4t} } }   {(4\pi t)^{\frac{d}{2}}}\in L^{p_1}(0,T;L^{p_2}(\mathbb{R}^{d})), \quad\forall~ p_1, p_2 \geq 1, ~~\frac{2}{dp_1}+\frac{1}{p_2}>1,
\end{eqnarray}
and
\begin{eqnarray}
\Gamma_t(x)=\nabla_x\gamma_t(x)\in L^{q_1}(0,T;L^{q_2}(\mathbb{R}^{d})), \quad\forall~ q_1, q_2 \geq 1, ~~\frac{2}{q_1}+\frac{d}{q_2}>d+1,
\end{eqnarray}
which provide the bound (with the choice $q_1=q_2=(\frac{d+2}{d+1})^-$)
\begin{eqnarray}
\|\nabla\rho_\varepsilon\|_{L^p(\Omega_T)}\leq C \|\Gamma\|_{L^{q_1}(\Omega_T)}\|Z^\varepsilon\|_{L^2(\Omega_T)}+\|\gamma\|_{L^p(0,T; L^1(\mathbb{R}^d))} \|\nabla\rho_0\|_{L^p(\mathbb{R}^d)}
\end{eqnarray}
for any $p\in [1, 2+\frac{4}{d})$. Combining \eqref{trinfiniforc}, one has
\begin{eqnarray}\nonumber
\|Z^\varepsilon\|_{L^p(\Omega_T)}&\leq& \|\nabla c^\varepsilon \|_{L^\infty(\Omega_T)} \| \nabla \rho_\varepsilon\|_{L^p(\Omega_T)}+\|\rho_\varepsilon\|_{L^\infty(\Omega_T)}\|\triangle c^\varepsilon \|_{L^p(\Omega_T)}\\
&\leq& C, \quad \quad \forall~ p\in [1, 2+\frac{4}{d}).
\end{eqnarray}
By a bootstrap argument of the regularity property of the heat equation, one has
\begin{eqnarray}
\|Z^\varepsilon\|_{L^p(\Omega_T)}\leq C, \quad \quad \forall~ p\in [1, d+1).
\end{eqnarray}
By the maximal regularity of heat equation in $L^p$ space (see \cite[Theorem X.12]{B1983}), we easily get
\begin{eqnarray}\label{regula2}
\|\rho_\varepsilon\|_{W^{2,1}_p(\Omega_T)}\leq C (\|\rho_0\|_{L^p(\Omega_T)}+\|Z^\varepsilon\|_{L^p(\Omega_T)})\leq C, \quad \forall~ p\in [1,d+1).
\end{eqnarray}
Based on \eqref{regula2}, using the Morrey's inequality, one has
\begin{eqnarray}
\|\rho_\varepsilon\|_{C^{0,\alpha}(\overline{\Omega}_T)}\leq C\|\rho_\varepsilon\|_{W^{1,p}(\Omega_T)}\leq C, \mbox{~for~some~} 0<\alpha<1.
\end{eqnarray}
Then, we also have
\begin{eqnarray}
\|\rho_\varepsilon\|_{C^{\alpha,\frac{\alpha}{2}}(\overline{\Omega}_T)}\leq C \|\rho_\varepsilon\|_{C^{0,\alpha}(\overline{\Omega}_T)} \leq C, \mbox{~for~some~} 0<\alpha<1,
\end{eqnarray}
where the space $C^{\alpha,\frac{\alpha}{2}}(\overline{\Omega}_T)$ is defined by
$$
C^{\alpha,\frac{\alpha}{2}}(\overline{\Omega}_T)=\{u\in C(\overline{\Omega}_T)| \sup\limits_{(t,x)\neq (s,y)\in\overline{\Omega}_T } \frac{|u(t,x)-u(s,y)|}{|t-s|^{\frac{\alpha}{2}}+|x-y|^{\alpha}}<+\infty \}
$$
with the norm
$$
\|u\|_{C^{\alpha,\frac{\alpha}{2}}(\overline{\Omega}_T)}= \sup\limits_{(t,x)\in\overline{\Omega}_T } |u(t,x)|+\sup\limits_{(t,x)\neq (s,y)\in\overline{\Omega}_T } \frac{|u(t,x)-u(s,y)|}{|t-s|^{\frac{\alpha}{2}}+|x-y|^{\alpha}}.
$$
By the classical Schauder estimates for the  heat equation (see \cite[Theorem X.13]{B1983}), one has
\begin{eqnarray}
\|c_\varepsilon\|_{C^{2+\alpha,1+\frac{\alpha}{2}}(\overline{\Omega}_T )}\leq C\big( \|\rho_\varepsilon\|_{C^{\alpha,\frac{\alpha}{2}}(\overline{\Omega}_T)}+ \|c_0\|_{C^{2+\alpha}(\mathbb{R}^{d} )}\big),
 \end{eqnarray}
 where
\begin{eqnarray}\nonumber
 \|c_\varepsilon\|_{C^{2+\alpha,1+\frac{\alpha}{2}}(\overline{\Omega}_T)}=\sum\limits_{ 0\leq2r+|s|\leq 2 } \| D^r_tD_x^sc_\varepsilon\|_{L^\infty(\overline{\Omega}_T)}
 +\|\partial_tc_\varepsilon\|_{C^{\alpha,\frac{\alpha}{2}}(\overline{\Omega}_T)}+\|D^2c_\varepsilon\|_{C^{\alpha,\frac{\alpha}{2}}(\overline{\Omega}_T)}
 \end{eqnarray}
By \eqref{heatequ}, one obtains
 \begin{eqnarray}
  \|\rho_\varepsilon\|_{C^{2+\alpha,1+\frac{\alpha}{2}}(\overline{\Omega}_T)}\leq C\big( \|c_\varepsilon\|_{C^{2+\alpha,1+\frac{\alpha}{2}}(\overline{\Omega}_T)}+ \|\rho_0\|_{C^{2+\alpha}(\mathbb{R}^{d})}\big),
 \end{eqnarray}
 which finishes the proof.
\end{proof}

\subsection{Compactness argument and  solvability of the intermediate PDEs}
Before giving the convergence results,  we  introduce the definition of weak solution to the parabolic-parabolic KS equation \eqref{diffaggKSeq} which we deal with throughout this paper. Indeed,
we ask for more regularities than needed for the definition and these regularities will be
proved in Theorem \ref{conthe}.
\begin{defn} \label{weaksolutionpde}(weak solution)
Given the initial data $(\rho_0(x), c_0(x))$ satisfying the conditions \eqref{inicon1} and \eqref{inicon2}. Let $T>0$, we shall say that $(\rho(t, x),c(t,x))$ is a weak solution to \eqref{diffaggKSeq} with the initial data $(\rho_0, c_0)$ if it satisfies:
\begin{itemize}
\item[1. ] Regularity:
\begin{eqnarray*}
\rho\in L^2\big(0,T; H^1(\mathbb{R}^{d})\big)\cap  L^\infty\big(0,T; L^2(\mathbb{R}^{d})\big), \quad \partial_t\rho\in L^2\big(0,T; H^{-1}(\mathbb{R}^{d})\big),\\
\int_{\mathbb{R}^{d}}|x|\rho (t,x) dx <\infty \quad\quad  \mbox{~for~all~} 0\leq t\leq T.\quad\quad\quad\quad\quad\quad
 \end{eqnarray*}
\item[2. ] For all $\varphi\in {C}_0^\infty(\mathbb{R}^d)$ and $0<t\leq T$, the following holds,
\begin{eqnarray}\label{Sindensityfun}\nonumber
&&\int_{\mathbb{R}^d}\rho(t,x)\varphi(x)dx-\int_{\mathbb{R}^d}\rho_0(x)\varphi(x)dx-\int_0^t\int_{\mathbb{R}^d}  \rho(s,x)\Delta \varphi(x) dxds\\
&&\qquad\qquad=\int_0^t\int_{\mathbb{R}^d}\rho(s,x)\nabla c(x)\cdot\nabla\varphi(x)dxds,
\end{eqnarray}
where $c$ is the  mild solution to the second equation of \eqref{diffaggKSeq}, i.e.
\begin{eqnarray}\label{cexpression}
c=e^{-\lambda t} e^{t\Delta} c_0+ \int_0^{t} e^{\lambda(s-t)} e^{(t-s)\Delta}\rho ds.
\end{eqnarray}
\end{itemize}
\end{defn}

Utilizing the uniform estimates obtained in Proposition \ref{uniest}, we have
the following convergence results. The proof will be given in the Appendix.
\begin{thm}\label{conthe}
Let $(\rho_\varepsilon, c_\varepsilon)$ be solutions to \eqref{intermeKSeq} with the initial data $(\rho_0, c_0)$. Assume that $(\rho_0, c_0)$ satisfy the conditions \eqref{inicon1} and \eqref{inicon2}, then there is a subsequence of $(\rho_\varepsilon, c_\varepsilon)$(without relabeling for convenience) and functions $(\rho, c)$ such that as $\varepsilon \rightarrow 0$
\begin{eqnarray}\label{comcon1}
\rho_\varepsilon \rightharpoonup \rho \mbox{~in~} L^\infty\big(0,T;  L^2(\mathbb{R}^d)\big)\\ \label{comcon2}
\nabla\rho_\varepsilon \rightharpoonup \nabla\rho \mbox{~in~} L^2\big(0,T; L^2(\mathbb{R}^d)\big)\\ \label{comcon3}
\rho_\varepsilon \rightarrow \rho \mbox{~in~} L^2\big(0,T; L^2(B_R)\big)  \mbox{~for~any~ball~} B_R \\ \label{comcon4}
c_\varepsilon \rightarrow c \mbox{~in~} L^\infty\big(0,T; L^2(B_R)\big) \mbox{~for~any~ball~} B_R
\end{eqnarray}
where $B_R$ is a ball centered at $0$  with radius $R$ and $(\rho, c)$ is a weak solution to \eqref{diffaggKSeq} with the following regularities:
\begin{eqnarray}\label{rhoregularity}
&&\rho\in L^2\big(0,T; H^1(\mathbb{R}^{d})\big)\cap  L^\infty\big(0,T; L^2(\mathbb{R}^{d})\big),\nonumber \\
&&\partial_t\rho\in L^2\big(0,T; H^{-1}(\mathbb{R}^{d})\big), \quad \int_{\mathbb{R}^{d}}|x|\rho (t,x) dx <\infty.
\end{eqnarray}
Furthermore, if $(\rho_0, c_0)$ also satisfy the condition \eqref{inicon3}, then the solution satisfies
\begin{eqnarray}\label{higeregularity}
\|\rho\|_{C^{2,1}(\overline{\Omega}_T)}+\| c\|_{C^{2,1}(\overline{\Omega}_T)}\leq C,
\end{eqnarray}
which means that this solution is a classical solution.
\end{thm}

\subsection{Estimate between $\rho_\varepsilon$ and $\rho$ in Kantorovich-Rubinstein distance}
We introduce the following mean-field self-consistent stochastic process $({X}_t)_{t\geq0}$:
\begin{eqnarray}\label{stoKSfin}\nonumber
X_t &=&X_0+ \int_0^t\int_0^{s}\int_{\mathbb{R}^{d}} \frac{e^{-\frac{|y|^2}{4 (s-r)} +\lambda (r-s) } }   {\big(4\pi(s-r)\big)^{\frac{d}{2}}}   \nabla\rho_r(X_s+y)dy drds\\
&&+ \int_0^te^{-\lambda s}\int_{\mathbb{R}^{d}}\frac{e^{-\frac{|{X}_s-y|^2}{4s} } }   {(4\pi s)^{\frac{d}{2}}}\nabla c_0(y)dyds+\sqrt{2}B_t, ~~\hspace{0.2cm} t\geq 0,
\end{eqnarray}
where we require $({X}_t)_{t\geq0}$ possessing a marginal density $(\rho_t)_{t\geq0}$ for any $t\geq 0$. By It\^{o}'s formula, we know that $\rho$ is a weak solution to the KS equation \eqref{diffaggKSeq}.  So we call \eqref{stoKSfin} the corresponding SDE of \eqref{diffaggKSeq}.

The drift term is
 \begin{eqnarray}
B_{[\rho]}(t,x)= \int_0^{t}\int_{\mathbb{R}^{d}} \frac{e^{-\frac{|y|^2}{4 (t-s)} +\lambda (s-t) } }   {\big(4\pi(t-s)\big)^{\frac{d}{2}}}   \nabla\rho_s(x+y)dy ds+e^{-\lambda t}\int_{\mathbb{R}^{d}}\frac{e^{-\frac{|x-y|^2}{4t} } }   {(4\pi t)^{\frac{d}{2}}}\nabla c_0(y)dy,
\end{eqnarray}
 which has the following properties:
 \begin{lem} \label{Hlpnoeps}
Assume that
$$\|\nabla c_0\|_{L^\infty(\mathbb{R}^{d})}+\|D^2c_0\|_{L^\infty(\mathbb{R}^{d})}+\|\nabla\rho\|_{L^\infty(\overline{\Omega}_T)}+\|D^2\rho\|_{L^\infty(\overline{\Omega}_T)}<+\infty,
$$
then
\begin{enumerate}
  \item $B_{[\rho]}(t,x)$ is continuous in $[0,T]\times\mathbb{R}^{d}$;
   \item  $|B_{[\rho]}(t,x)|\leq C$,  for all $(t,x)\in [0,T]\times \mathbb{R}^{d}$, where $C$ depends only on $\|\nabla c_0\|_{L^\infty(\mathbb{R}^{d})}$ and $\|\nabla\rho\|_{L^\infty(\overline{\Omega}_T)}$;
  \item $B_{[\rho]}(t,x)$ is Lipschitz with respect to x,
i.e.  there exists a constant $L$ depending only on $\|D^2c_0\|_{L^\infty(\mathbb{R}^{d})}$ and $\|D^2\rho\|_{L^\infty(\overline{\Omega}_T)}$, such that for $t\in[0,T], x_1, x_2 \in \mathbb{R}^{d}$ ,
$$
\big|B_{[\rho]}(t,x_1)-B_{[\rho]}(t,x_2) |\leq{L\big|x_1-x_2|}.
$$
\end{enumerate}
\end{lem}
The proof is simple and similar with Lemma \ref{Hlp}, We omit the process here.
\begin{defn}
For any fixed $T>0$,  initial data ${X}_0$ and given probability space $\big(\Omega, \mathcal{F}, \mathbb{P}\big)$ endowed with a $d$-dimensional $(\mathcal{F}_t)_{t\in[0,T]}$-Brownian motion $({B}_t)_{t\in[0,T]}$, if there is a stochastic process $({X}_t)_{t\in[0,T]}$ adapted to $(\mathcal{F}_t)_{t\in[0,T]}$ and it has a time marginal density $\rho$ (satisfying $\|\nabla\rho\|_{L^\infty(\overline{\Omega}_T)}+\|D^2\rho\|_{L^\infty(\overline{\Omega}_T)}<+\infty$) such that $\big({X}_t,\rho_t\big)_{t\in [0,T]}$ satisfies \eqref{stoKSfin} almost surely (a.s.) in the probability space $\big(\Omega, \mathcal{F}, (\mathcal{F}_t)_{t\geq0}, \mathbb{P}\big)$  for all $t\in[0, T]$, we say that $\big({X}_t,\rho_t\big)_{t\geq0}$  is a global strong solution to \eqref{stoKSfin}.
\end{defn}

\begin{thm}\label{existunifina}
Suppose ${X}_0$ is a random variable with the density  $\rho_0$ and $(\rho_0, c_0)$ satisfy the conditions \eqref{inicon1}, \eqref{inicon2}, \eqref{inicon3}. Then there exists a unique global strong solution to \eqref{stoKSfin}.
\end{thm}
\begin{proof}
For any given Brownian motion $({B}_t)_{t\in[0,T]}$ in \eqref{stoKSfin},  equation \eqref{sto1equ} is equivalent to the following equation:
\begin{eqnarray}\label{sto1equhigh}\nonumber
\bar{X}_t^{\varepsilon} &=&X_0+\int_0^t\int_0^{s-\varepsilon}\int_{\mathbb{R}^{d}} \frac{e^{-\frac{|y|^2}{4 (s-r)} +\lambda (r-s) } }   {\big(4\pi(s-r)\big)^{\frac{d}{2}}}   \nabla\rho^\varepsilon_r(\bar{X}^{\varepsilon}_s+y)dy drds\\
&&+ \int_0^te^{-\lambda s}\int_{\mathbb{R}^{d}}\frac{e^{-\frac{|\bar{X}_s^{\varepsilon}-y|^2}{4s} } }   {(4\pi s)^{\frac{d}{2}}}\nabla c_0(y)dyds+\sqrt{2}B_t, ~~\hspace{0.2cm} t\geq 0.
\end{eqnarray}
By It\^{o}'s formula, one knows that the time marginal density ${\rho}^\varepsilon(t,x)$ is a weak solution to \eqref{intermeKSeq}. From Proposition \ref{higeuniest}, one knows that
\begin{eqnarray}
\|\rho^\varepsilon\|_{C^{2,1}(\overline{\Omega}_T)}+\| c^\varepsilon\|_{C^{2,1}(\overline{\Omega}_T)}\leq C.
\end{eqnarray}
From Theorem \ref{conthe}, one knows that there is a subsequence of $(\rho_\varepsilon, c_\varepsilon)$(without relabeling for convenience) and functions $(\rho, c)$ such that as $\varepsilon \rightarrow 0$,  $(\rho_\varepsilon, c_\varepsilon)$ converge to $(\rho, c)$ in the sense of
\eqref{comcon1}--\eqref{comcon4}. One also has
\begin{eqnarray}
\|\rho\|_{C^{2,1}(\overline{\Omega}_T)}+\| c\|_{C^{2,1}(\overline{\Omega}_T)}\leq C.
\end{eqnarray}
Taking $\rho$ into the function $B_{[\rho]}(t,x)$ in Lemma \ref{Hlpnoeps}, one knows that $B_{[\rho]}(t,x)$ is bounded and Lipschitz continuous. Then there exists a unique stochastic process $({X}_t)_{t\geq0}$ such that
\begin{eqnarray}\label{stoKSfinno}\nonumber
X_t &=&X_0+ \int_0^t\int_0^{s}\int_{\mathbb{R}^{d}} \frac{e^{-\frac{|y|^2}{4 (s-r)} +\lambda (r-s) } }   {\big(4\pi(s-r)\big)^{\frac{d}{2}}}   \nabla\rho_r(X_s+y)dy drds\\
&&+ \int_0^te^{-\lambda s}\int_{\mathbb{R}^{d}}\frac{e^{-\frac{|{X}_s-y|^2}{4s} } }   {(4\pi s)^{\frac{d}{2}}}\nabla c_0(y)dyds+\sqrt{2}B_t, ~~\hspace{0.2cm} t\geq 0,
\end{eqnarray}
and $({X}_t)_{t\geq0}$ admits a time marginal  density denoted by $\tilde{\rho}(t,x)$ (see \cite[Theorem 9.1.9]{SV1979}).

For any $\varphi(x)\in C_b^2(\mathbb{R}^d)$, the It\^{o}'s formula states that
\begin{eqnarray}\label{itoresult1}\nonumber
\varphi(X_t)&=&\varphi(X_0)+\int_0^t\nabla\varphi(X_s)\cdot B_{[\rho]}\big(s,{X}_s\big)ds\\
&&+\sqrt{2}\int_0^t\nabla\varphi(X_s)\cdot d{B}_s+\int_0^t\triangle \varphi(X_s)ds.
\end{eqnarray}
Taking  expectation of \eqref{itoresult1}, $\tilde{\rho}$ is a weak solution to the following linear Fokker-Planck equation:
\begin{eqnarray}\label{regkineticequs11}
\left\{\begin{array}{l}
\partial_t\tilde{\rho}(t,x)=\triangle \tilde{\rho}(t,x)-\nabla \cdot[B_{[\rho]}(t,x)\tilde{\rho}(t,x)], \quad x\in\mathbb{R}^d, \quad t>0,\\
\tilde{\rho}(0,x)=\rho_0(x).
\end{array}\right.
\end{eqnarray}
Since $\rho$  is a weak solution to \eqref{diffaggKSeq}  by Theorem \ref{conthe}, then it is also a weak solution to \eqref{regkineticequs11}. And the weak solution of \eqref{regkineticequs11}  in the class of $L^\infty\big(0,T; L^2(\mathbb{R}^{d})\big)\cap L^2\big(0,T; H^1(\mathbb{R}^{d})\big)$  is unique, then
$$\tilde{\rho}=\rho,$$
which means that $({X}_t, \rho(t,x))_{t\geq0}$ is a strong solution to \eqref{stoKSfin}.

Now we prove the uniqueness of strong solution to \eqref{stoKSfin}. Assume that $({X}_t, \rho_t)_{t\geq0}, (\bar{X}_t,\bar{\rho}_t)_{t\geq0}$ are two strong solutions to  \eqref{stoKSfin} with the same initial data and Brownian motion. Then
\begin{eqnarray}\label{uniq}
\bar{X}_t-{X}_t=\int_0^t\big(B_{[\bar{\rho}]}(s,\bar{X}_s)-B_{[\rho]}(s,{X}_s)\big)ds.
\end{eqnarray}
Taking expectation of \eqref{uniq}, one has
\begin{eqnarray}\label{uniquenessst}
\mathbb{E}[\sup\limits_{t\in[0,T]}|\bar{X}_t-{X}_t|]\leq \int_0^T\mathbb{E}\big[|B_{[\bar{\rho}]}(s,\bar{X}_s)-B_{[\rho]}(s,{X}_s)|\big]ds.
\end{eqnarray}
One also has
\begin{eqnarray}\label{uniqueness11}\nonumber
&&|B_{[\bar{\rho}]}(s,\bar{X}_s)-B_{[\rho]}(s,{X}_s)|\\ \nonumber
&\leq& \big|\int_0^{s}\int_{\mathbb{R}^{d}} \frac{e^{-\frac{|y|^2}{4 (s-r)} +\lambda (r-s)  } }   {\big(4\pi(s-r)\big)^{\frac{d}{2}}}   (\nabla \bar{\rho}_r(\bar{X}_s+y)- \nabla \bar{\rho}_r(X_s+y))dy dr\big|\\ \nonumber
&&+\big|\int_0^{s}\int_{\mathbb{R}^{d}} \frac{e^{-\frac{|y|^2}{4 (s-r)} +\lambda (r-s)  } }   {\big(4\pi(s-r)\big)^{\frac{d}{2}}}   (\nabla \bar{\rho}_r(X_s+y)- \nabla\rho_r(X_s+y))dy dr\big|\\ \nonumber
&&+\int_{\mathbb{R}^{d}}\frac{e^{-\frac{|y|^2}{4s} } }   {(4\pi s)^{\frac{d}{2}}}|\nabla c_0(y+\bar{X}_s)-\nabla c_0(y+{X}_s)|dy\\ \nonumber
&\leq&  \big(\|D^2c_0\|_{L^\infty(\mathbb{R}^{d})}+\|D^2\bar{\rho}\|_{L^\infty(\overline{\Omega}_T)}s\big) |\bar{{X}}_s-{X}_s|\\
&&+\big|\int_0^{s}\int_{\mathbb{R}^{d}} \frac{e^{-\frac{|y|^2}{4 (s-r)} +\lambda (r-s)  } }   {\big(4\pi(s-r)\big)^{\frac{d}{2}}}   \big(\nabla \bar{\rho}_r(X_s+y)- \nabla\rho_r(X_s+y)\big)dy dr\big|.
\end{eqnarray}
Denote $I=:\big|\int_0^{s}\int_{\mathbb{R}^{d}} \frac{e^{-\frac{|y|^2}{4 (s-r)} +\lambda (r-s)  } }   {\big(4\pi(s-r)\big)^{\frac{d}{2}}}   \big(\nabla \bar{\rho}_r(X_s+y)- \nabla\rho_r(X_s+y)\big)dy dr\big|$.
Plugging \eqref{uniqueness11} into \eqref{uniquenessst}, one has
\begin{eqnarray}\label{uniqueness12}
\mathbb{E}[\sup\limits_{t\in[0,T]}|\bar{X}_t-{X}_t|]\leq  C_T\int_0^T\mathbb{E}[\sup\limits_{\tau\in[0,s]}|\bar{X}_\tau-{X}_\tau|+I]ds.
\end{eqnarray}
Suppose $(\bar{{Y}_t}; {Y}_t) $ is an independent copy
of $(\bar{{X}}_t; {X}_t)$, we have
\begin{eqnarray}\label{uniqueness13}\nonumber
\mathbb{E}[I]&=&\mathbb{E}_x\big[\big|\int_0^{s}\int_{\mathbb{R}^{d}} \frac{e^{-\frac{|{X}_s-y|^2}{4 (s-r)} +\lambda (r-s) } }   {\big(4\pi(s-r)\big)^{\frac{d}{2}}}  \frac{y-{X}_s}{2(s-r)} (\bar{\rho}_r(y)-\rho_r(y))dy dr\big|\big]\\ \nonumber
&=&\mathbb{E}_x\mathbb{E}_y\big[\big|\int_0^{s} \Big(\frac{e^{-\frac{|{X}_s-\bar{Y}_r|^2}{4 (s-r)} +\lambda (r-s) } }   {\big(4\pi(s-r)\big)^{\frac{d}{2}}}  \frac{\bar{Y}_r-{X}_s}{2(s-r)} -\frac{e^{-\frac{|{X}_s-{Y}_r|^2}{4 (s-r)} +\lambda (r-s) } }   {\big(4\pi(s-r)\big)^{\frac{d}{2}}}  \frac{{Y}_r-{X}_s}{2(s-r)}\Big) dr\big|\big]\\ \nonumber
&=&\mathbb{E}_y\big[\big|\int_0^{s}\int_{\mathbb{R}^{d}} \Big(\frac{e^{-\frac{|x-\bar{Y}_r|^2}{4 (s-r)} +\lambda (r-s) } }   {\big(4\pi(s-r)\big)^{\frac{d}{2}}}  \frac{\bar{Y}_r-x}{2(s-r)} -\frac{e^{-\frac{|x-{Y}_r|^2}{4 (s-r)} +\lambda (r-s) } }   {\big(4\pi(s-r)\big)^{\frac{d}{2}}}  \frac{{Y}_r-x}{2(s-r)}\Big)\rho_s(x) dxdr\big|\big]\\ \nonumber
&=&\mathbb{E}_y\big[\big|\int_0^{s}\int_{\mathbb{R}^{d}} \frac{e^{-\frac{|x|^2}{4 (s-r)} +\lambda (r-s) } }   {\big(4\pi(s-r)\big)^{\frac{d}{2}}}  \big( \nabla\rho_r(\bar{Y}_r+x) -\nabla\rho_r({Y}_r+x)\big)
 dxdr\big|\big]\\ \nonumber
&\leq&\|D^2\rho\|_{L^\infty(\overline{\Omega}_T)}\mathbb{E}\big[\int_0^{s} | \bar{Y}_r-{Y}_r|dr\big] \\
&\leq& \|D^2\rho\|_{L^\infty(\overline{\Omega}_T)}\int_0^{s}\mathbb{E}[\sup\limits_{\tau\in[0,r]}|\bar{{X}}_\tau-{X}_\tau|]dr.
\end{eqnarray}
Plugging \eqref{uniqueness13} into \eqref{uniqueness12}, one has
\begin{eqnarray}\label{uniqueness14}
\mathbb{E}[\sup\limits_{t\in[0,T]}|\bar{X}_t-{X}_t|]\leq  C_T\int_0^T\mathbb{E}[\sup\limits_{\tau\in[0,s]}|\bar{X}_\tau-{X}_\tau|]ds.
\end{eqnarray}
By $\mathbb{E}[|\bar{X}_0-{X}_0|]=0$ and the Gronwall's Lemma, we obtain that $\mathbb{E}[\sup\limits_{t\in[0,T]}|\bar{X}_t-{X}_t|]\equiv0$ and
\begin{eqnarray}\label{uniqenesssdetrong}
\mathcal{M}_T(\bar{\rho},\rho)\leq \mathbb{E}[\sup\limits_{t\in[0,T]}|\bar{X}_t-{X}_t|]=0.
\end{eqnarray}
Therefore $\bar{\rho}=\rho$ and ${X}_t=\bar{X}_t$ a.s. for all $t\geq0$.
\end{proof}
According to the proof of existence of strong solution to \eqref{stoKSfin} and It\^{o}'s formula, we have the following corollary.
\begin{cor}\label{relationsdepde}
Suppose ${X}_0$ is a random variable with the density  $\rho_0$ and $(\rho_0, c_0)$ satisfy the conditions \eqref{inicon1}, \eqref{inicon2}, \eqref{inicon3}. Then the relationship between the weak solution to \eqref{diffaggKSeq} and the strong solution to \eqref{stoKSfin} can be expressed:
 \begin{enumerate}[(i)]
\item If $\rho(t,x)$ is a weak solution to \eqref{diffaggKSeq} with the initial data $\rho_0(x)$, then  for any random variable $X_0$ with the density $\rho_0$, there is a unique process $(X_t)_{t\geq0}$ with the density $\rho$ such that $(X_t, \rho_t)_{t\geq0}$ is a strong solution to \eqref{stoKSfin} with the initial data $(X_0, \rho_0)$.
\item If $(X_t, \rho_t)_{t\geq0}$ is a strong solution to  \eqref{stoKSfin} with the initial data $(X_0, \rho_0)$, then $\rho$ is a weak solution to \eqref{diffaggKSeq} with the initial data $\rho_0$.
 \end{enumerate}
\end{cor}

\begin{cor}\label{uniquepde}
Suppose $(\rho_0, c_0)$ satisfy the conditions \eqref{inicon1}, \eqref{inicon2} and \eqref{inicon3}. Then the weak solution to \eqref{diffaggKSeq} is unique.
\end{cor}
\begin{proof}
Suppose $\rho$, $\bar{\rho}$ are two weak solutions to \eqref{diffaggKSeq} with the same initial data $\rho_0$. For any fixed random variable $X_0$ with the density $\rho_0$, by the Corollary \ref{relationsdepde} $(i)$, there exists two processes $(X_t)_{t\geq0}$ and $(\bar{X}_t)_{t\geq0}$ such that $(X_t, \rho_t)_{t\geq0}$ and  $(\bar{X}_t, \bar{\rho}_t)_{t\geq0}$ both are strong solutions to \eqref{stoKSfin}  with the same initial data $(X_0,\rho_0)$. Thus \eqref{uniqenesssdetrong} holds, which gives the uniqueness of \eqref{diffaggKSeq}.
\end{proof}

\begin{rem}
Recently, the uniqueness  of  weak solution  to \eqref{diffaggKSeq} has been concerned by many scholars. In fact, a uniqueness result is established in the class of solutions $\rho\in C([0,T];L^1_2(\mathbb{R}^{2}))\cap L^\infty((0,T)\times \mathbb{R}^{2})$, see \cite{CLM2014}. Carrapatoso and Mischler have improved the uniqueness result: if the initial data
with finite mass $M < 8\pi$, finite second log-moment and finite entropy, uniqueness of the `free energy' solution in $\mathbb{R}^{2}$ is known, see \cite{CMS2014}. Furthermore, when the initial data belong to critical scaling-invariant Lebesgue spaces,  the uniqueness of integral solutions and the uniqueness of self-similar solutions have been analyzed in \cite{CEM2014}. In this paper, we prove the uniqueness of  weak solution  to \eqref{diffaggKSeq} by utilizing the strong solution of \eqref{stoKSfin} as a characteristic line.
\end{rem}

 Next, we compare the trajectories between \eqref{stoKSfin} and \eqref{sto1equhigh}.
\begin{prop}\label{fincoup}
Let $({X}_t, \rho_t)_{t\geq0}$ and $(\bar{X}^{\varepsilon}_t, \rho^\varepsilon_t)_{t\geq0}$ be the unique strong solution to \eqref{stoKSfin} and \eqref{sto1equhigh} respectively,  with the same  initial data ${X}_0$ and Brownian motion $({B} _t)_{t\geq0}$. Then for any  $T>0$,
\begin{eqnarray}\label{sentwoestimates}
\mathbb{E}\big[\sup\limits_{t\in[0,T]}|\bar{{X}}^{\varepsilon}_t-{X}_t|\big]\leq  C_Te^{C_T T} \varepsilon.
\end{eqnarray}
\end{prop}
\begin{proof}
Subtracting one equation from the other one, one has
\begin{eqnarray}\label{distan23}\nonumber
&&\sup\limits_{\tau\in[0,t]}|\bar{{X}}^{\varepsilon}_\tau-{X}_\tau|\\ \nonumber
 &\leq& \big|\int_0^t\int_0^{s}\int_{\mathbb{R}^{d}} \frac{e^{-\frac{|y|^2}{4 (s-r)} +\lambda (r-s)  } }   {\big(4\pi(s-r)\big)^{\frac{d}{2}}}   (\nabla\rho^\varepsilon_r(\bar{X}^{\varepsilon}_s+y)- \nabla\rho^\varepsilon_r(X_s+y))dy drds\big|\\ \nonumber
&&+\big|\int_0^t\int_0^{s}\int_{\mathbb{R}^{d}} \frac{e^{-\frac{|y|^2}{4 (s-r)} +\lambda (r-s)  } }   {\big(4\pi(s-r)\big)^{\frac{d}{2}}}   (\nabla\rho^\varepsilon_r(X_s+y)- \nabla\rho_r(X_s+y))dy drds\big|
 \\ \nonumber
&&+\int_0^t\int_{s-\varepsilon}^{s}\int_{\mathbb{R}^{d}} \frac{e^{-\frac{|y|^2}{4 (s-r)}  } }   {\big(4\pi(s-r)\big)^{\frac{d}{2}}}   |\nabla\rho^\varepsilon_r(\bar{X}^{\varepsilon}_s+y)|dy drds\\ \nonumber
&&+\int_0^t\int_{\mathbb{R}^{d}}\frac{e^{-\frac{|y|^2}{4s} } }   {(4\pi s)^{\frac{d}{2}}}|\nabla c_0(y+{X}_s)-\nabla c_0(y+\bar{X}_s^{\varepsilon})|dyds\\ \nonumber
&\leq&  \|\nabla\rho^\varepsilon\|_{L^\infty(\overline{\Omega}_T)}\varepsilon t+\big(\|D^2c_0\|_{L^\infty(\mathbb{R}^{d})}+\|D^2\rho^\varepsilon\|_{L^\infty(\overline{\Omega}_T)}t\big) \int_0^t|\bar{{X}}^{\varepsilon}_s-{X}_s|ds\\
&&+\big|\int_0^t\int_0^{s}\int_{\mathbb{R}^{d}} \frac{e^{-\frac{|y|^2}{4 (s-r)} +\lambda (r-s)  } }   {\big(4\pi(s-r)\big)^{\frac{d}{2}}}   \big(\nabla\rho^\varepsilon_r(X_s+y)- \nabla\rho_r(X_s+y)\big)dy drds\big|.
\end{eqnarray}
Denote $I=:\big|\int_0^{s}\int_{\mathbb{R}^{d}} \frac{e^{-\frac{|y|^2}{4 (s-r)} +\lambda (r-s)  } }   {\big(4\pi(s-r)\big)^{\frac{d}{2}}}   \big(\nabla\rho^\varepsilon_r(X_s+y)- \nabla\rho_r(X_s+y)\big)dy dr\big|$. Taking expectation of \eqref{distan23}, one has
\begin{eqnarray}\label{distan231}
\mathbb{E}[\sup\limits_{\tau\in[0,t]}|\bar{{X}}^{\varepsilon}_\tau-{X}_\tau|] &\leq&    C_T\varepsilon+C_T\int_0^t\mathbb{E}[\sup\limits_{\tau\in[0,s]}|\bar{{X}}^{\varepsilon}_\tau-{X}_\tau|+I]ds.
\end{eqnarray}
Similarly with the computation of \eqref{uniqueness13}, one has
\begin{eqnarray}\label{distan232}
\mathbb{E}[I]\leq \|D^2\rho\|_{L^\infty(\overline{\Omega}_T)}\int_0^s\mathbb{E}[\sup\limits_{\tau\in[0,r]}|\bar{{X}}^{\varepsilon}_\tau-{X}_\tau|]dr.
\end{eqnarray}
Plugging \eqref{distan232} into \eqref{distan231}, one has
\begin{eqnarray}\label{distan233}
\mathbb{E}[\sup\limits_{\tau\in[0,t]}|\bar{{X}}^{\varepsilon}_\tau-{X}_\tau|] &\leq&    C_T\varepsilon+C_T\int_0^t\mathbb{E}[\sup\limits_{\tau\in[0,s]}|\bar{{X}}^{\varepsilon}_\tau-{X}_\tau|]ds.
\end{eqnarray}
By the Gronwall's inequality, one obtains that
\begin{eqnarray}\label{distan233}
\mathbb{E}[\sup\limits_{\tau\in[0,t]}|\bar{{X}}^{\varepsilon}_\tau-{X}_\tau|] \leq     C_Te^{C_T t}\varepsilon,
\end{eqnarray}
which finishes the proof.

\end{proof}

\section{Rigorous derivation of the propagation of chaos}
First, we recall the following standard equivalent notions of propagation of chaos from the lecture notes of Sznitman \cite[Proposition 2.2]{S1991}.
\begin{defn}\label{paropagationchaos}
Let $E$ be a polish space.   A sequence of symmetric probability measures  $f^N$ on $E^N$ are said to be $f$-chaotic, $f$ is a probability measure on $E$, if one of three following equivalent condition is satisfied:
\begin{enumerate}[(i)]
 \item The sequence of second marginals $f^{2,N}\rightharpoonup f\otimes f$
 as $N\rightarrow\infty$;
 \item For all $j\geq1$, the sequence of j-th marginals $f^{j,N}\rightharpoonup f^{\otimes j }$  as $N\rightarrow\infty$;
 \item The empirical measure $\frac{1}{N}\sum\limits_{i=1}^N\delta_{{{X}}^{i, N}}$ (${X}^{i, N}$, $i=1,\cdots, N$ are canonical coordinates on $E^N$)  converge  in law to the constant random variable $f$ as $N\rightarrow\infty$.
     \end{enumerate}
\end{defn}

Finally,  we finish the proof of Theorem \ref{meanfieldres} by proving $(ii)$ of notions of propagation of chaos.

{\bf Proof of Theorem \ref{meanfieldres}:}
By Theorem \ref{existunim},
  there exists a unique global strong solution $\bar{X}_t^{i,\varepsilon}$ to \eqref{sto1equhigh} with the initial data ${X}^{i} _0$ and Brownian motion ${B}^{i} _t$ for any $1\leq i\leq N$. Moreover $\{\bar{X}_t^{i,\varepsilon}\}_{i=1}^N$ are i.i.d. and their common density $\rho^\varepsilon$ is the unique weak solution to \eqref{intermeKSeq}.

By Theorem \ref{existunifina}, there exists a unique global strong solution $(X_t^{i}, \rho^i)$ to \eqref{stoKSfin} with the initial data ${X}^{i} _0$ and Brownian motion ${B}^{i} _t$ for any $1\leq i\leq N$. $\{{X}_t^{i}\}_{i=1}^N$ are i.i.d. and their common density $\rho$ is the unique weak solution to  \eqref{diffaggKSeq} by Corollary  \ref{relationsdepde} and Corollary  \ref{uniquepde}.


Combining Proposition \ref{P} and Proposition \ref{fincoup}, for any $T>0$ and $1\leq i\leq N$, one has
\begin{eqnarray}\label{coupmeth}\nonumber
\mathbb{E}[\sup\limits_{t\in[0,T]}|X_t^{i,\varepsilon}-{X}^i_t|] &\leq& \mathbb{E}[\sup\limits_{t\in[0,T]}|X_t^{i,\varepsilon}-\bar{X}_t^{i,\varepsilon}|]+\mathbb{E}[\sup\limits_{t\in[0,T]}|\bar{X}_t^{i,\varepsilon}-{X}^i_t|]\\
&\leq &   \frac{C}{{\sqrt {N}\varepsilon^{\frac{d}{2}}}}\exp\big({\frac{C}{\varepsilon^{\frac{d}{2}+1}}}\big)+ C_Te^{C_T t}\varepsilon.
\end{eqnarray}
Choosing $\varepsilon=\varepsilon(N)= \lambda(\ln N)^{-\frac{2}{d+2}}\rightarrow0$ as $N\rightarrow\infty$, where $\lambda$ is a large enough positive constant such that $C\lambda^{-\frac{d+2}{2}}<\frac{1}{2}$,  then
\begin{eqnarray}\label{strongvertion11}\nonumber
\mathbb{E}[\sup\limits_{t\in[0,T]}|X_t^{i,\varepsilon(N)}-{X}^i_t|] &\leq& C\lambda^{-\frac{d}{2}}N^{C\lambda^{-\frac{d+2}{2}}-\frac{1}{2}}(\ln N)^{\frac{d}{d+2}}+ C_Te^{C_T t}\varepsilon\\
&& \rightarrow0 \quad\quad \mbox{~as~} N\rightarrow\infty.
\end{eqnarray}

Denote by $\tilde{F}_t(x_1,\cdots,x_N,\hat{x}_1,\cdots,\hat{x}_j)$ the joint marginal  distribution of \\ $\big({X}^{1,\varepsilon}_t,\cdots,{X}^{N,\varepsilon}_t,{{X}}^{1}_t,\cdots,{{X}}^{j}_t\big)_{t\geq0}$, then one has the following facts
\begin{eqnarray*}
&&f_t^{\otimes j}=\int_{\mathbb{R}^{Nd}}\tilde{F}_t(dx_1,\cdots,dx_N,\cdot),\\
&&f_t^{(j),\varepsilon}=\int_{\mathbb{R}^{Nd}}\tilde{F}_t(\cdot,dx_{j+1},\cdots,dx_N,d\hat{x}_1,\cdots,d\hat{x}_j),\\ &&\int_{\mathbb{R}^{(N-j)d}}\tilde{F}_t(\cdot,dx_{j+1},\cdots,dx_N,\cdot)\in \Lambda(f_t^{(j),\varepsilon}, f_t^{\otimes j}),
\end{eqnarray*}
where $f$ is the common distribution of $\big\{({X}^{i}_t)_{t\geq0}\big\}^N_{i=1}$.

By the exchangeability of $\{({X}^{i, \varepsilon}_t-{X}^{i}_t)_{t\geq0}\}_{i=1}^N$ ( similar to \eqref{exchange}), we obtain
\begin{eqnarray}\label{propagation} \nonumber
&&\mathcal {M}_T(f_t^{(j),\varepsilon}, f_t^{\otimes j})\\ \nonumber
&\leq& \nonumber \sup\limits_{t\in[0,T]}\int_{\mathbb{R}^{2jd}}\big(|x_1-\hat{x}_1|+\cdots+|x_j-\hat{x}_j|\big)\int_{\mathbb{R}^{(N-j)d}}\tilde{F}_t(dx_1,\cdots,dx_N,d\hat{x}_1,\cdots,d\hat{x}_j)\\ \nonumber
&=&j\sup\limits_{t\in[0,T]}\int_{\mathbb{R}^{(N+j)d}}|x_1-\hat{x}_1|\tilde{F}_t(dx_1,\cdots,dx_N,d\hat{x}_1,,\cdots,d\hat{x}_j)\\
&\leq&j \mathbb{E}_{x_1,\cdots,x_N,\hat{x}_1,\cdots,\hat{x}_j}\big[\sup\limits_{t\in[0,T]}|{X}^{1,\varepsilon}_t-{X}^{1}_t|\big].
\end{eqnarray}
Combining \eqref{strongvertion11} and  \eqref{propagation},
 one knows that $f_t^{(j),\varepsilon}$ narrowly converges to $f_t^{ \otimes j}$ by Proposition \ref{property},
 which means that \eqref{mean} holds true. \quad $\square$

\section{Appendix}
{\bf Proof of  Lemma \ref{energeyest}:}
\begin{proof}
First, multiplying  the first equation of \eqref{intermeKSeq} with $r\rho_\varepsilon ^{r-1}$, $r\geq 2 $ and integrating in $\mathbb{R}^d$, one has
\begin{eqnarray}\label{diK}
\frac{d}{dt}\int_{\mathbb{R}^d}\rho_\varepsilon ^rdx+\frac{4(r-1) }{r}\int_{\mathbb{R}^d}\big|\nabla (\rho_\varepsilon^{\frac{r}{2}})\big|^2dx
= (r-1)\int_{\mathbb{R}^d}\nabla \rho_\varepsilon^r \cdot\nabla c_\varepsilon dx.
\end{eqnarray}
Therefore, we just need to estimate the right hand of \eqref{diK}.
\begin{eqnarray}
\int_{\mathbb{R}^d}\nabla \rho_\varepsilon^r \cdot\nabla c_\varepsilon dx\leq\int_{\mathbb{R}^d}| \rho_\varepsilon^r \triangle c_\varepsilon |dx\leq \|\rho_\varepsilon(t)\|^r_{L^{r+1}(\mathbb{R}^d)}\|\triangle c_\varepsilon(t)\|_{L^{r+1}(\mathbb{R}^d)}.
\end{eqnarray}
Using Young's inequality and \eqref{estpara3}, one has
\begin{eqnarray}\nonumber\label{lowdest}
\int_0^t\int_{\mathbb{R}^d}\nabla \rho_\varepsilon^r \cdot\nabla c_\varepsilon dxdt&\leq& \frac{r}{r+1}\int_0^t\|\rho_\varepsilon(s)\|^{r+1}_{L^{r+1}(\mathbb{R}^d)}ds+\frac{1}{r+1} \int_0^t \|\triangle c_\varepsilon(s)\|^{r+1}_{L^{r+1}(\mathbb{R}^d)}ds\\
&\leq& C_{r,d}\big( \int_0^t\|\rho_\varepsilon(s)\|^{r+1}_{L^{r+1}(\mathbb{R}^d)}ds+t\|c_0 \|_{W^{2,{r+1}}(\mathbb{R}^d)}\big)
\end{eqnarray}
 where $C_{r,d}$ is a positive constant depending on $r$ and $d$.

Integrating \eqref{diK} in time interval and combining \eqref{lowdest}, we have
\begin{eqnarray}\label{senconddiK}\nonumber
\|\rho_\varepsilon\|^r_{L^r(\mathbb{R}^d)}+\frac{4(r-1) }{r}\int_0^t\| \nabla (\rho_\varepsilon^{\frac{r}{2}})\|^2_{L^2(\mathbb{R}^d)}ds
&\leq& \|\rho_0\|^r_{L^r(\mathbb{R}^d)}+C_{r,d}\big(t\|c_0 \|_{W^{2,r+1}(\mathbb{R}^d)}\\
&&+\int_0^t\|\rho_\varepsilon(s)\|^{r+1}_{L^{r+1}(\mathbb{R}^d)} ds\big).
\end{eqnarray}
To estimate the last term of \eqref{senconddiK}, taking $q_2=r+1, q_1=1, m=1$  in Lemma \ref{GNinequaliy}, one knows that
\begin{eqnarray}\label{estofGN}
\|\rho_\varepsilon\|^{r+1}_{L^{r+1}(\mathbb{R}^d)}\leq C_{r,d}\|\rho_\varepsilon\|^{(1-\theta)(r+1)}_{L^{1}(\mathbb{R}^d)}\| \nabla (\rho_\varepsilon^{\frac{r}{2}})\|^{\frac{2\theta(r+1)}{r}}_{L^2(\mathbb{R}^d)},
\end{eqnarray}
where
\begin{eqnarray}\nonumber
\theta=
\left\{\begin{array}{ll}
\frac{r^2}{(r+1)^2}, & d=1, \\
\frac{r}{r+1}, & d=2 ,\\
\frac{r^2}{(r+1)\big(\frac{2}{d}-1+r\big) },  &d\geq 3, r\geq \frac{d}{2}-1.
\end{array}\right.
\end{eqnarray}

Next, we split into three cases to deal with the inequality \eqref{senconddiK}.

{\bf Case1, $d=1$ :}  Plugging \eqref{estofGN} into \eqref{senconddiK} for $d=1$, one has
\begin{eqnarray}\label{senconddiKfor1d}\nonumber
\|\rho_\varepsilon\|^r_{L^r(\mathbb{R})}+\frac{4(r-1) }{r}\int_0^t\| \nabla (\rho_\varepsilon^{\frac{r}{2}})\|^2_{L^2(\mathbb{R})}ds
&\leq& \|\rho_0\|^r_{L^r(\mathbb{R})}+C_{r,1}\big(t\|c_0 \|_{W^{2,r+1}(\mathbb{R})}\\
&&+\int_0^t\|\rho_\varepsilon\|^{\frac{2r+1}{r+1}}_{L^{1}(\mathbb{R})} \|\nabla (\rho_\varepsilon^{\frac{r}{2}})\|^{\frac{2r}{r+1}}_{L^{2}(\mathbb{R})}ds\big).
\end{eqnarray}
Since $\frac{2r}{r+1}<2$, by Young's inequality,
\begin{eqnarray}\label{senconddiKfor1d11}\nonumber
\|\rho_\varepsilon\|^r_{L^r(\mathbb{R})}+\frac{4(r-1) }{r}\int_0^t\| \nabla (\rho_\varepsilon^{\frac{r}{2}})\|^2_{L^2(\mathbb{R})}ds
&\leq& \|\rho_0\|^r_{L^r(\mathbb{R})}+C_{r,1}\big(t\|c_0 \|_{W^{2,r+1}(\mathbb{R})}+\int_0^t\|\rho_\varepsilon\|^{2r+1}_{L^{1}(\mathbb{R})}ds\big)\\
&&+\frac{2(r-1) }{r}\int_0^t\|\nabla (\rho_\varepsilon^{\frac{r}{2}})\|^{2}_{L^{2}(\mathbb{R})}ds.
\end{eqnarray}
Combining the mass conservation of the system, i.e. $\| \rho_\varepsilon(\cdot,t) \|_{L^1} =\| \rho_0(\cdot) \|_{L^1} = M_0$, we get the following uniform estimate for $d=1$:
\begin{eqnarray}\label{uniformford1}
\|\rho_\varepsilon\|^r_{L^r(\mathbb{R})}+\frac{2(r-1) }{r}\int_0^t\| \nabla (\rho_\varepsilon^{\frac{r}{2}})\|^2_{L^2(\mathbb{R})}ds
\leq \|\rho_0\|^r_{L^r(\mathbb{R})}+C_{r,1}t\big(\|c_0 \|_{W^{2,r+1}(\mathbb{R})}+ M_0^{2r+1}\big).
\end{eqnarray}

{\bf Case2, $d=2$ :}  Plugging \eqref{estofGN} into \eqref{senconddiK} for $d=2$, one has
\begin{eqnarray}\label{uniformford2}\nonumber
\|\rho_\varepsilon\|^r_{L^r(\mathbb{R}^2)}+\frac{4(r-1) }{r}\int_0^t\| \nabla (\rho_\varepsilon^{\frac{r}{2}})\|^2_{L^2(\mathbb{R}^2)}ds
&\leq& \|\rho_0\|^r_{L^r(\mathbb{R}^2)}+C_{r,2}\big(t\|c_0 \|_{W^{2,r+1}(\mathbb{R}^2)}\\
&&+ M_0 \int_0^t\| \nabla (\rho_\varepsilon^{\frac{r}{2}})\|^2_{L^2(\mathbb{R}^2)}ds\big).
\end{eqnarray}
Therefore, for small $M_0$, we can get a uniform estimate.

{\bf Case3, $d\geq 3$ :} Plugging \eqref{estofGN} into \eqref{senconddiK} for $d\geq 3$, one has
\begin{eqnarray}\label{senconddiKhighd}\nonumber
\|\rho_\varepsilon\|^r_{L^r(\mathbb{R}^d)}+\frac{4(r-1) }{r}\int_0^t\| \nabla (\rho_\varepsilon^{\frac{r}{2}})\|^2_{L^2(\mathbb{R}^d)}ds
&\leq& \|\rho_0\|^r_{L^r(\mathbb{R}^d)}+C_{r,d}\big(t\|c_0 \|_{W^{2,r+1}(\mathbb{R}^d)}\\
&&+M_0^{(1-\theta)(r+1)}  \int_0^t\| \nabla (\rho_\varepsilon^{\frac{r}{2}})\|^{\frac{2\theta(r+1)}{r}}_{L^2(\mathbb{R}^d)} ds\big).
\end{eqnarray}
One can easily check that $\frac{2\theta(r+1)}{r}>2$, so it fails to get a uniform estimate by the same method with $d=1, 2$. We will use the method applied by \cite{CP2006} to get the uniform estimate. We come back to deal with the right hand of \eqref{diK}, i.e.
\begin{eqnarray}\label{diKrighthand}
\int_{\mathbb{R}^d}\nabla \rho_\varepsilon^r \cdot\nabla c_\varepsilon dx= 2\int_{\mathbb{R}^d} \rho_\varepsilon^{\frac{r}{2}} \nabla \rho_\varepsilon^{\frac{r}{2}} \cdot\nabla c_\varepsilon dx.
\end{eqnarray}
Using H\"{o}lder's inequality, Sobolev injection and \eqref{estpara2}, one has
\begin{eqnarray}\label{righthandest}
\int_{\mathbb{R}^d} |\rho_\varepsilon^{\frac{r}{2}} \nabla \rho_\varepsilon^{\frac{r}{2}} \cdot\nabla c_\varepsilon| dx&\leq&
\|\rho_\varepsilon^{\frac{r}{2}}\|_{L^{\frac{2d}{d-2}}(\mathbb{R}^d)} \|\nabla \rho_\varepsilon^{\frac{r}{2}}\|_{L^2(\mathbb{R}^d)}
\|\nabla c_\varepsilon\|_{L^d(\mathbb{R}^d)} \\ \nonumber
&\leq& C_d \|\nabla \rho_\varepsilon^{\frac{r}{2}}\|^2_{L^2(\mathbb{R}^d)}\big(\|\nabla c_0\|_{L^d(\mathbb{R}^d)} +C_{d,q} \sup\limits_{0\leq s\leq t} \|\rho_\varepsilon(\cdot,s)\|_{L^{q}(\mathbb{R}^d)}\big),
\end{eqnarray}
where $\frac{d}{2}< q \leq d$.

Combining \eqref{diK}, \eqref{diKrighthand} and \eqref{righthandest}, one obtains that
\begin{eqnarray}\label{disestforhighd11}
\frac{d}{dt}\|\rho_\varepsilon\|^r_{L^{r}(\mathbb{R}^d)}\leq 2(r-1)\|\nabla \rho_\varepsilon^{\frac{r}{2}}\|^2_{L^2(\mathbb{R}^d)}\big(-\frac{2}{r}+C_d\|\nabla c_0\|_{L^d(\mathbb{R}^d)} +C_{d,q} \sup\limits_{0\leq s\leq t} \|\rho_\varepsilon(\cdot,s)\|_{L^{q}(\mathbb{R}^d)}\big),
\end{eqnarray}
which finishes the proof.
\end{proof}

{\bf Proof of  Proposition \ref{uniest}:}

\begin{proof}
Taking $r=2$ in \eqref{disestfor1d} and \eqref{disestfor2d} respectively, one obtains \eqref{domainest} for $d=1, 2$,  immediately.

When $d\geq3$, we chose $r=q=a$, $a\in(\frac{d}{2}, d]$ in \eqref{disestforhighd}. If the initial data $\|\nabla c_0\|_{L^d(\mathbb{R}^d)}+\|\rho_0\|_{L^{a}(\mathbb{R}^d)}$ is small enough compared to $\frac{1}{a}$, then the right hand of \eqref{disestforhighd} is negative and thus $\|\rho_\varepsilon\|^a_{L^{a}(\mathbb{R}^d)}$ decreases and this remains true for all times, i.e. we have
\begin{eqnarray}\label{aenergyest}
\sup\limits_{0\leq s\leq T} \|\rho_\varepsilon(s)\|_{L^{a}(\mathbb{R}^d)}\leq \|\rho_0\|_{L^{a}(\mathbb{R}^d)}.
\end{eqnarray}

Then choosing $r=2$, $q=a$, $a\in(\frac{d}{2}, d]$ in \eqref{disestforhighd} and integrating in time interval, one has
\begin{eqnarray}\label{2energyest}
\|\rho_\varepsilon\|^2_{L^{2}(\mathbb{R}^d)}+ 2\big(1-C_d\|\nabla c_0\|_{L^d(\mathbb{R}^d)} -C_{d,a}  \|\rho_0\|_{L^{a}(\mathbb{R}^d)}\big)\int_0^T\|\nabla \rho_\varepsilon\|^2_{L^2(\mathbb{R}^d)}ds\leq \|\rho_0\|^2_{L^{2}(\mathbb{R}^d)},
\end{eqnarray}
i.e. \eqref{domainest} also holds true for high dimensions as long as $C_d\|\nabla c_0\|_{L^d(\mathbb{R}^d)} +C_{d,a}  \|\rho_0\|_{L^{a}(\mathbb{R}^d)}< 1$.

The uniform estimates of $\|\rho_\varepsilon\|_{L^\infty(0,T; L^\infty(\mathbb{R}^{d}))}$ is similar with the parabolic-elliptic KS equation, (see \cite[Theorem 4.2.]{BLZ2014} and \cite[Theorem 2.2.]{LY2019}).  We omit the process here.

Now, we estimate the time derivative. Multiplying  the first equation of \eqref{intermeKSeq} with a test function $\varphi(x)\in {C}_0^\infty(\mathbb{R}^d)$, then integrating in space, one has
\begin{eqnarray}\label{diKfortime}
\int_{\mathbb{R}^d}\partial_t\rho_\varepsilon \varphi dx
=-\int_{\mathbb{R}^d}\big(\nabla \rho^\varepsilon  - \rho^\varepsilon\nabla c^\varepsilon\big)\cdot\nabla\varphi dx.
\end{eqnarray}
Then
\begin{eqnarray}\label{diKfortimed1}
\Big|\int_{\mathbb{R}^d}\partial_t\rho_\varepsilon\varphi dx\Big|
\leq \|\nabla\varphi\|_{L^2({\mathbb{R}^d})} \big(\|\nabla\rho_\varepsilon\|_{L^2({\mathbb{R}^d})}+\|\rho_\varepsilon\|_{L^2({\mathbb{R}^d})} \|\nabla c_\varepsilon\|_{L^\infty({\mathbb{R}^d})}\big).
\end{eqnarray}
By \eqref{estpara2} and \eqref{Linfinity}, one has
\begin{eqnarray}\label{diKfortimed11}
\|\nabla c_\varepsilon(t)\|_{L^\infty({\mathbb{R}^d})}\leq \|\nabla c_0\|_{L^\infty({\mathbb{R}^d})}+ C\|\rho_\varepsilon \|_{L^\infty(0,T; L^{\infty}(\mathbb{R}^d))}\leq C.
\end{eqnarray}
Plugging \eqref{diKfortimed11} into \eqref{diKfortimed1},  one has
\begin{eqnarray}\label{diKfortimed12}
\frac{|\int_{\mathbb{R}^d}\partial_t\rho_\varepsilon\varphi dx|}{\|\nabla\varphi\|_{L^2({\mathbb{R}^d})}}
\leq \|\nabla\rho_\varepsilon\|_{L^2({\mathbb{R}^d})}+C\|\rho_\varepsilon\|_{L^2({\mathbb{R}^d})}.
\end{eqnarray}
Combining \eqref{domainest} and \eqref{diKfortimed12}, one obtains \eqref{timeest}.

Finally, we prove the uniform estimate of first moment of ${\rho}_\varepsilon$.
Recalling that ${\rho}_\varepsilon(t,x)$ is the common time marginal density of i.i.d. processes $\big\{({\bar{X}}^{i,\varepsilon}_t)_{t\geq0}\big\}^N_{i=1}$  and  $\big\{({\bar{X}}^{i,\varepsilon}_t)_{t\geq0}\big\}^N_{i=1}$ satisfy equation \eqref{sto1equ}. Then taking expectation of \eqref{sto1equ}, one obtains that for any $0\leq t\leq T$,
\begin{eqnarray}\label{momentest}
\int_{\mathbb{R}^d}|x|{\rho}_\varepsilon(t,x) dx=\mathbb{E}[|{\bar{X}}^{i,\varepsilon}_t|]\leq\mathbb{E}[|X_0|]+\int_0^t\int_{\mathbb{R}^d}|\nabla c_\varepsilon(s,x)|{\rho}_\varepsilon(s,x)dxds+\sqrt{2t}.
\end{eqnarray}
Using \eqref{Linfinity} and \eqref{diKfortimed11}, the second term of the right hand of \eqref{momentest} is controlled by
\begin{eqnarray}\label{sectermest}
\int_0^T\int_{\mathbb{R}^d}|\nabla c_\varepsilon|{\rho}_\varepsilon dxds
\leq
\|\nabla c_\varepsilon\|_{L^\infty(0,T; L^\infty(\mathbb{R}^d))} \|{\rho}_\varepsilon\|_{L^1(0,T; L^1(\mathbb{R}^d))}\leq CTM_0.
\end{eqnarray}
Plugging \eqref{sectermest} into \eqref{momentest}, one obtains \eqref{1moment} immediately.
\end{proof}

{\bf Proof of  Theorem \ref{conthe}:}
\begin{proof}
{\bf Step1:} We give the convergence of $\rho_\varepsilon$.
Based on the uniform estimates in Proposition \ref{uniest},  there exists a constant C independent of $\varepsilon$ such that
    $$\sup\limits_{0\leq t\leq T} \|\rho_\varepsilon\|^2_{L^2(\mathbb{R}^d)}\leq C, \quad \int_0^T\|\nabla\rho_\varepsilon\|_{ L^2(\mathbb{R}^{d})}^2\,dt\leq C, \quad \int_0^T\|\partial_t\rho_\varepsilon\|_{H^{-1}(\mathbb{R}^{d})}^2\,dt\leq C,
     $$
then there exists a subsequence $\rho_\varepsilon$ without relabeling such that
\eqref{comcon1} and \eqref{comcon2} hold.  Notice that the following compact embedding holds: for any ball $B_R$,
$$
H^1(B_R)\hookrightarrow\hookrightarrow L^2(B_R)\hookrightarrow H^{-1}(B_R).
$$

By the Lions-Aubin lemma \cite{CJL2014,L1982} and combining with the regularities, one arrives at
$$
\rho_\varepsilon \mbox{ is compact in } L^2(0,T; L^2\big(B_R)\big).
$$
Consequently, there exists a subsequence $\rho_\varepsilon$ without relabeling such that
\eqref{comcon3} holds. Furthermore, by the uniform estimates in Proposition \ref{uniest}, the regularity \eqref{rhoregularity} of $\rho$ follows.

{\bf Step2:} We give the convergence of $c_\varepsilon$.  Define
\begin{eqnarray}
c=e^{-\lambda t} e^{t\Delta} c_0+ \int_0^{t} e^{\lambda(s-t)} e^{(t-s)\Delta}\rho ds.
\end{eqnarray}
For any ball $B_R$, by the Bochner theorem in \cite [pp. 650]{E2002} and the Young's inequality for the convolution  \cite[pp. 99]{LL2001}, one has
\begin{eqnarray}\label{cconvergence}\nonumber
&&\|c_\varepsilon(t)-c(t)\|_{L^2(B_R)}\\
\nonumber&\leq& \|\int_0^{t} e^{\lambda(s-t)} e^{(t-s)\Delta}(\rho^\varepsilon-\rho) ds\|_{L^2(B_R)}+\|\int_{t-\varepsilon}^{t} e^{\lambda(s-t)} e^{(t-s)\Delta}\rho^\varepsilon ds\|_{L^2(B_R)}\\ \nonumber
&\leq&C\int_0^{t} e^{\lambda(s-t)} \|\rho^\varepsilon(\cdot,s)-\rho(\cdot,s)\|_{L^2(B_R)} ds+C\|\rho^\varepsilon\|_{L^\infty(0,T;L^2(B_R))}\int_{t-\varepsilon}^{t} e^{\lambda(s-t)} ds\\
&\leq&C\sqrt{t} \|\rho^\varepsilon-\rho\|_{L^2(0,T;L^2(B_R))}+C\|\rho^\varepsilon\|_{L^\infty(0,T;L^2(\mathbb{R}^d))}\varepsilon.
\end{eqnarray}
Therefore, one gets the convergence of \eqref{comcon4}.

{\bf Step3:}
Now, we prove that $(\rho, c)$ is exactly a weak solution to \eqref{diffaggKSeq}. For any test function $\varphi(x)\in C_0^\infty(\mathbb{R}^{d})$, one has the following equation
\begin{eqnarray}\label{regularweakform}\nonumber
&&\int_{\mathbb{R}^d}\rho_\varepsilon(t,\cdot)\varphi(x)dx-\int_{\mathbb{R}^d}\rho_0\varphi(x)dx-\int_0^t\int_{\mathbb{R}^d}  \rho_\varepsilon(t,x)\Delta  \varphi dxds\\
&=&\int_0^t\int_{\mathbb{R}^d}\rho_\varepsilon(t,x) \nabla c_\varepsilon(t,x)\cdot\nabla\varphi(x)dxds.
\end{eqnarray}
By the weak convergence of $\rho_{\varepsilon}(t,x)$, the linear parts converge as follows
\begin{eqnarray}\label{comconver2}
\int_{\mathbb{R}^d}\varphi(x)\rho_{\varepsilon}(t,x)dx\rightarrow \int_{\mathbb{R}^d}\varphi(x)\rho(t,x)dx, \mbox{~as~} \varepsilon\rightarrow 0.
\end{eqnarray}
\begin{eqnarray}\label{comconver3}
\int_0^t\int_{\mathbb{R}^d} \rho_\varepsilon(t,x) \Delta  \varphi dxds \rightarrow \int_0^t\int_{\mathbb{R}^d} \rho(t,x)\Delta \varphi dxds, \mbox{~as~} \varepsilon\rightarrow 0.
\end{eqnarray}
The nonlinear part is divided as follows
\begin{eqnarray}\label{comconver4}
&&\big|\int_0^t\int_{\mathbb{R}^d}\big[\rho_\varepsilon(t,x)\nabla c_\varepsilon -\rho(t,x)\nabla c\big]\cdot\nabla\varphi(x)dxds\big|\\ \nonumber
&\leq& \big|\int_0^t\int_{\mathbb{R}^d}\big[\rho_\varepsilon(t,x)-\rho(t,x)\big]\nabla c_\varepsilon\cdot\nabla\varphi(x)dxds\big| +\big|\int_0^t\int_{\mathbb{R}^d}\rho\big[\nabla c_\varepsilon-\nabla c \big]\cdot\nabla\varphi(x)dxds\big|
\end{eqnarray}
From \eqref{Linfiniforc},  the first term of \eqref{comconver4} can be estimated by
\begin{eqnarray}\label{comconver5}
 \big|\int_0^t\int_{\mathbb{R}^d}\big[\rho_\varepsilon(t,x)-\rho(t,x)\big]\nabla c_\varepsilon\cdot\nabla\varphi(x)dxds\big|\leq C \| \rho_\varepsilon-\rho\|_{L^2(0,T; L^2(supp\varphi ))}
\end{eqnarray}
where $C$ is independent of $\varepsilon$. Hence the first term goes to zero by \eqref{comcon3} as $\varepsilon$ goes to zero. Furthermore,  one can easily check that the second term of \eqref{comconver4}  goes to zero by \eqref{comcon4}, i.e.
\begin{eqnarray}\label{comconver6}
\big|\int_0^t\int_{\mathbb{R}^d}\rho\big[\nabla c_\varepsilon-\nabla c \big]\cdot\nabla\varphi dxds\big|\leq C \| \nabla c_\varepsilon-\nabla c\|_{L^2(0,T; L^2(supp\varphi ))} \rightarrow 0 \mbox{~as~} \varepsilon\rightarrow 0.
\end{eqnarray}
Combining \eqref{comconver4}, \eqref{comconver5} and \eqref{comconver6} together, the nonlinear part converges too. So we prove that $(\rho, c)$ is a weak solution to \eqref{diffaggKSeq}.
\end{proof}

{\bf Acknowledgments}
The research of R. Yang was partially supported by National Natural Science Foundation of China  No. 11601021. The research of S. Wang is supported by the NSFC (11831003, 11771031, 11531010) of China and NSF of Qinghai Province (2017-ZJ-908).



{\small
\begin{thebibliography}{}
\bibitem{BLZ2014}
S.~Bian, J.-G.~Liu, and C.~Zou,
Ultra-contractivity for KS equations with diffusion exponent $m>1-2/d$,
Kinet.~Relat.~Mod., {\bf 7} (2014), 9-28.

\bibitem{BCC2}
F.~Bolley, J.~A.~Ca\~{n}izo and J.~A.~Carrillo, Stochastic mean-field limit: non-lipschitz forces and swarming,
Math.~Models Methods Appl.~Sci., {\bf 21} (2011), 2179-2210.

\bibitem{B1983}
H.~Brezis, Analyse fonctionnelle. Collection Math\'{e}matiques Appliqu\'{e}es pour la Ma\^{i}trise. [Collection of Applied
Mathematics for the Master's Degree], Masson, Paris, 1983.


\bibitem{CMS2014}
K.~Carrapatoso and S.~Mischler,  Uniqueness and long time asymptotics for the parabolic-parabolic keller-segel equation, Archive for Rational Mechanics and Analysis, {\bf 42}(2), (2014), 291-345.

\bibitem{CLM2014}
 J.~A.~Carrillo, S.~Lisini and E.~Mainini,  Uniqueness for Keller-Segel-type chemotaxis models, Discrete Contin.~Dyn.~Syst.~34, {\bf 4}, (2014), 1319-1338.

\bibitem{CJL2014}
X. Chen, A. J\"ungel, J.-G.~Liu, A note on Aubin-Lions-Dubinski lemma, Acta Appl.~Math., {\bf 133}, (2014), 33-43.

\bibitem{CEM2014}
L.~Corrias, M.~Escobedo and J.~Matos, Existence, uniqueness and asymptotic behavior of the solutions to the fully
parabolic Keller-Segel system in the plane, J.~Differential Equations, {\bf 257}, 6 (2014), 1840-1878.


\bibitem{CP2006}
L.~Corrias and B.~Perthame, Critical space for the parabolic-parabolic Keller-Segel model in $\mathbb{R}^d)$, C.~R.~Acad.~ Sci.~Paris, Ser. ~I, 342, (2006), 745-750.

\bibitem{EO2007}
R.~Erban and H.~G.~Othmer,  Taxis equations for amoeboid cells, J.~Math.~Biol, {\bf 54}, 6 (2007), 847-885.

\bibitem{E2002}
 L.~C.~Evans, Partial Differential Equations, Graduate Studies in Mathematics, vol. 19. AMS, Providence (2002).


\bibitem{JTT2017}
J.~F.~Jabir, D.~Talay and  M.~Tomasevic, Mean--field limit of a particle approximation of the one-dimensional parabolic--parabolic Keller-Segel model without smoothing, Electron.~Commun.~Prob, {\bf 23}, (2017).

\bibitem{JLW2020}
A. Jungel, O. Leingang, S. Wang, Vanishing cross-diffusion limit in a Keller¨CSegel system with
additional cross-diffusion, Nonlinear Analysis, 192(2020), 111698:1-21.

\bibitem{IY2012}
S.~Ishida and T.~Yokota,
Global existence of weak solutions to quasilinear degenerate
Keller-Segel systems of parabolic-parabolic type, J.~Differential Equations, {\bf 252} (2012), 1421-1440.

\bibitem{LSU1968}
 O.~A.~Lady\v{z}enskaja, V.~A.~Solonnikov and  N.~N.~Ural'ceva, Linear and Quasilinear Equations of Parabolic Type, Amer.~Math.~Soc., Providence, RI, (1968).

\bibitem{LL2001}
E.~H.~Lieb and M.~Loss, Analysis. Graduate Studies in Mathematics, vol. 14, 2nd edn. American Mathematical Society,
Providence (2001).

\bibitem{L1982}
P.-L.~Lions, Sym\'{e}trie et compacit\'{e} dans les espaces de sobolev. J.~Funct.~Anal. {\bf 49} (1982), 315-334.

\bibitem{LY2019}
J.-G.~Liu and R.~Yang, Propagation of chaos for the Keller-Segel equation with a logarithmic cut-off,  preprint, (2019)

\bibitem{OH2002}
H.~G.~Othmer and T.~Hillen,  The diffusion limit of transport equations.~II.~Chemotaxis equations, SIAM
J.~Appl.~Math. {\bf 62}, 4 (2002), 1222-1250.

\bibitem{P2004}
 B.~Perthame, PDE models for chemotactic movements: parabolic, hyperbolic and kinetic, Appl.~Math. {\bf 49}, 6 (2004),
539-564.

\bibitem{SV1979}
D.~W.~Stroock and S.~R.~S.~Varadhan, Multidimensional Diffusion Processes, Springer, New York, 1979.

\bibitem{S2006}
Y.~Sugiyama, Global existence in sub-critical cases and finite time blow-up in super-critical cases to degenerate Keller-Segel systems. Differ.~Integral Equ., {\bf 19} (2006), 841-876.

\bibitem{SK2006}
 Y.~Sugiyama and H.~Kunii, Global existence and decay properties for a degenerate Keller-Segel model with a power factor in
drift term, J.~Differential Equations, {\bf 227} (2006), 333-364.


\bibitem{S1991}
A.-S.~Sznitman,
Topics in propagation of chaos,
In Ecole d'Et\'{e} de Probabilit\'{e}s de Saint-Flour XIX-1989,
Lecture Notes in Math. 1464. Springer, Berlin, (1991).

\bibitem{V}
C.~Villani,
Optimal Transport, Old and New, Grundlehren Math. Wiss. 338, Springer, Berlin, (2009).

\bibitem{WLC2019}
J.~Wang, Y.~Li and L.~Chen,
Supercritical degenerate parabolic-parabolic Keller-Segel system: existence criterion given by the best constant in Sobolev's inequality, Z.~Angew.~Math.~Phys. (2019) 70:71.

\end{thebibliography}}
\end{document}